\documentclass[12pt]{amsart}

\usepackage{amsthm,amsfonts,amsmath,amssymb,latexsym,epsfig,mathrsfs,yfonts,marvosym,latexsym,epsfig}
\usepackage[usenames]{color}

\DeclareMathAlphabet\oldmathcal{OMS}        {cmsy}{b}{n}
\SetMathAlphabet    \oldmathcal{normal}{OMS}{cmsy}{m}{n}
\DeclareMathAlphabet\oldmathbcal{OMS}       {cmsy}{b}{n}
 
\usepackage{eucal}

\usepackage{graphicx}
\usepackage[all]{xy}
\usepackage{epsfig}
\usepackage{layout}
\usepackage[active]{srcltx}

\newtheorem{theorem}{Theorem}[section]
\newtheorem{lemma}[theorem]{Lemma}
\newtheorem{proposition}[theorem]{Proposition}
\newtheorem{corollary}[theorem]{Corollary}
\newtheorem{definition}[theorem]{Definition}
\newtheorem{definition/proposition}[theorem]{Definition/Proposition}
\newtheorem{question}{Question}
\newtheorem{conjecture}[theorem]{Conjecture}
\newenvironment{example}{\medskip \refstepcounter{theorem}
\noindent  {\bf Example \thetheorem}.\rm}{\,}
\newenvironment{remark}{\medskip \refstepcounter{theorem}
\noindent  {\bf Remark \thetheorem}.\rm}{\,}
\newenvironment{remarks}{\medskip \refstepcounter{theorem}
\noindent  {\bf Remarks \thetheorem}.\rm}{\,}
\newtheorem{ack}{Acknowledgments} 

\def\BOne{{\mathchoice {\rm 1\mskip-4mu l} {\rm 1\mskip-4mu l}
                          {\rm 1\mskip-4.5mu l} {\rm 1\mskip-5mu l}}}
\def\Ric{{\rm Ric}}

\def\fract#1#2{\raise4pt\hbox{$ #1 \atop #2 $}}
\def\decdnar#1{\phantom{\hbox{$\scriptstyle{#1}$}}
\left\downarrow\vbox{\vskip15pt\hbox{$\scriptstyle{#1}$}}\right.}

\def\bbc{{\mathbb C}}

\def\bbp{{\mathbb P}}
\def\bbq{{\mathbb Q}}
\def\bbr{{\mathbb R}}

\def\bbt{{\mathbb T}}

\def\bbz{{\mathbb Z}}

\def\gra{\alpha}

\def\grd{\delta}

\def\grg{\gamma}
\def\gri{\iota}
\def\grk{\kappa}
\def\grl{\lambda}

\def\gro{\omega}

\def\grr{\rho}

\def\grt{\tau}

\def\grz{\zeta}
\def\grD{\Delta}

\def\grG{\Gamma}

\def\grL{\Lambda}

\def\grS{\Sigma}

\def\bfa{{\bf a}}

\def\bfl{{\bf l}}

\def\bfp{{\bf p}}

\def\bfu{{\bf u}}
\def\bfv{{\bf v}}
\def\bfw{{\bf w}}

\def\bfz{{\bf z}}

\def\bfS{{\boldsymbol S}}

\def\cala{{\mathcal A}}

\def\calc{{\mathcal C}}
\def\calo{{\mathcal O}}

\def\cald{{\mathcal D}}
\def\cale{{\mathcal E}}
\def\calf{{\mathcal F}}

\def\calh{{\mathcal H}}
\def\cali{{\mathcal I}}
\def\calj{{\mathcal J}}
\def\calk{{\mathcal K}}

\def\calm{{\mathcal M}}

\def\calo{{\mathcal O}}

\def\calr{{\mathcal R}}
\def\cals{{\oldmathcal S}}

\def\calw{{\mathcal W}}

\def\calz{{\oldmathcal Z}}

\def\la#1{\hbox to #1pc{\leftarrowfill}}
\def\ra#1{\hbox to #1pc{\rightarrowfill}}

\def\ga{{\mathfrak a}}

\def\gc{{\mathfrak c}}

\def\ge{{\mathfrak e}}
\def\gf{{\mathfrak f}}

\def\gh{{\mathfrak h}}
\def\gi{{\mathfrak i}}

\def\gl{{\mathfrak l}}
\def\gm{{\mathfrak m}}
\def\gn{{\mathfrak n}}
\def\go{{\mathfrak o}}
\def\gp{{\mathfrak p}}

\def\gr{{\mathfrak r}}
\def\gs{{\mathfrak s}}
\def\gt{{\mathfrak t}}
\def\gu{{\mathfrak u}}

\def\gy{{\mathfrak y}}

\def\gA{{\mathfrak A}}
\def\gB{{\mathfrak B}}
\def\gC{{\mathfrak C}}
\def\gD{{\mathfrak D}}

\def\gF{{\mathfrak F}}
\def\gG{{\mathfrak G}}
\def\gH{{\mathfrak H}}

\def\gM{{\mathfrak M}}

\def\gP{{\mathfrak P}}
\def\gQ{{\mathfrak Q}}
\def\gR{{\mathfrak R}}
\def\gS{{\mathfrak S}}
\def\gT{{\mathfrak T}}

  \def\bfV{\mbox{{\bf V}}} \def\bfS{\mbox{{\bf S}}}    \def\bfH{\mbox{{\bf H}}}      \def\bfF{\mbox{{\bf F}}}

\def\hook{\mathbin{\hbox to 6pt{%
                 \vrule height0.4pt width5pt depth0pt
                 \kern-.4pt
                 \vrule height6pt width0.4pt depth0pt\hss}}}

\def\Sas{\mbox{Sas}}

\def\dif{\gD\gi\gf\gf(M)}

\begin{document}
\bibliographystyle{amsalpha}

\title{Contact Structures of Sasaki Type and their Associated Moduli}\thanks{The author was supported by grant \#519432 from the Simons Foundation.}

\author{Charles P. Boyer}
\address{Department of Mathematics and Statistics,
University of New Mexico, Albuquerque, NM 87131.}

\email{cboyer@unm.edu} 

\keywords{Contact structure of Sasaki type, Sasaki moduli, K-stability, Sasaki-Einstein, CSC, extremal Sasaki metric}

\subjclass[2000]{Primary: 53D35; Secondary:  53C25}

\date{\today}

\dedicatory{Dedicated to David Blair on the occasion of his 78th Birthday}

\begin{abstract}
This article is based on a talk at the RIEMain in Contact conference in Cagliari, Italy in honor of the 78th birthday of David Blair one of the founders of modern Riemannian contact geometry. The present article is a survey of a special type of Riemannian contact structure known as Sasakian geometry. 
An ultimate goal of this survey is to understand the moduli of classes of Sasakian structures as well as the moduli of extremal and constant scalar curvature Sasaki metrics, and in particular the moduli of Sasaki-Einstein metrics.
\end{abstract}

\maketitle

\markboth{Contact Structures of Sasaki Type}{Charles P. Boyer}

\tableofcontents

\section{Introduction}

It is the purpose of this paper to survey much of what is known about a special type of contact structure, namely, a Sasakian structure. I concentrate on results obtained since the publication of the book \cite{BG05}. Moreover, results in the book are placed in the context of Sasaki moduli spaces. 

Roughly, Sasakian geometry is to contact geometry what K\"ahlerian geometry is to symplectic geometry.
A contact structure $\cald$ is said to be of {\it Sasaki type} if there is a Sasakian structure $\cals$ whose contact 1-form $\eta$ satisfies $\cald=\ker\eta$. Equivalently, the foliation $\calf_R$ described by the Reeb vector field $R$ is K\"ahlerian and $R$ lies in $\ga\gu\gt(\cals)$, the Lie algebra of the group of Sasaki automorphisms. Moreover, the affine cone $(C(M),I)$ associated to a Sasaki manifold $M$ has a natural exact K\"ahlerian structure.
Thus, Sasaki geometry, sandwiched naturally between two K\"ahler geometries, is considered to be the odd dimensional sister to K\"ahler geometry. There are, however, substantial differences. The fiducial example of a compact K\"ahler manifold is $\bbc\bbp^n$ with its unique, up to biholomorphism, K\"ahlerian complex structure. In contrast there are exotic contact structures of Sasaki type on the sphere $S^{2n+1}$ when $n\geq 2$. So the standard fiducial Sasakian structure on $S^{2n+1}$ is one of many, in fact infinitely many. The category of Sasaki manifolds $\cals\calm$ is more closely related to the category of K\"ahler orbifolds, not just manifolds. Moreover, there is a multiplication in $\cals\calm$ mimicking the product of K\"ahler orbifolds, but much more complicated. The multiplication in $\cals\calm$ is the {\it Sasaki join} operation.

From the existence of the affine cone $(C(M),I)$ and a theorem of Rossi, cf. Theorem 5.60 of \cite{CiEl12} in dimension $>3$ and Marinescu and Yeganefar \cite{MaYe07} in dimension 3, it follows \cite{BMvK15} that a contact structure of Sasaki type is holomorphically fillable (likewise, a K-contact structure is symplectically fillable). So in both cases we have a special type of contact metric structure as described in the lectures of David Blair and in more detail in his books \cite{Bla76a,Bl02,Bla10}.

We note also that a Sasakian structure is a strictly pseudoconvex CR structure whose Levi form is K\"ahlerian.
We give a list of some fundamental open problems concerning contact structures of Sasaki type in order to motivate our discussion.

\begin{enumerate}
\item Given a smooth manifold $M$ determine how many inequivalent contact structures $\cald$ of Sasaki type there are:

\begin{itemize}
\item with distinct first Chern class $c_1(\cald)$.
\item with the same first Chern class $c_1(\cald)$.
\end{itemize}
\item Given a contact structure or isotopy class of contact structures:

\begin{itemize}
\item Determine  the space of compatible Sasakian structures.
\item Determine the (pre)-moduli space of Sasaki classes.
\item Determine the (pre)-moduli space of extremal Sasakian structures.
\item Determine the  (pre)-moduli space of Sasaki-Einstein or $\eta$-Einstein structures.
\item Determine the (pre)-moduli space of Sasakian structures with the same underlying CR structure.
\item Determine those having distinct underlying CR structures within the same  isotopy class of contact structures.
\end{itemize}
\end{enumerate}
A complete answer to all of these questions is probably intractible; nevertheless, it seems judicious to have answering such problems as an ultimate goal.

\begin{ack}
This survey is based on recent joint work with various colleagues: Hongnian Huang, Eveline Legendre, Leonardo Macarini, Justin Pati, Christina T{\o}nnesen-Friedman, Craig van Coevering, and Otto van Koert, to whom I owe a large dept of gratitude. I also want to thank Miguel Abreu, Dieter Kotschick, Leonardo Macarini, and Otto van Koert for correspondence on certain pertinent issues presented here. Finally, I thank the organizers, Gianluca Bande, Beniamino Cappelletti Montano and
Paola Piu, for a wonderful conference as well as the many participants with whom I enjoyed numerous conversations.
\end{ack}

\section{Basics of Sasaki Geometry}
Here we give only a brief description of Sasakian structures, and refer to \cite{BG05} for details and further description.  A contact structure on a $2n+1$ dimensional manifold $M$ is a codimension one distribution $\cald$ that is maximally non-integrable in the sense that there is a 1-form $\eta$ on $M$ such that $\eta\wedge (d\eta)^n\neq 0$ everywhere on $M$. We shall always assume that our contact manifolds $M$ are compact without boundary (closed). A {\it contact manifold} $M$ with contact bundle $\cald$ will be denoted by $(M,\cald)$. When we choose a contact 1-form $\eta$ we have a {\it strict contact manifold} $(M,\eta)$. We shall at times refer to the vector bundle $\cald$ as a {\it contact structure} and the 1-form $\eta$ as a {\it strict contact structure}. We denote by $\gC(M)$ the set of all oriented and co-oriented contact structures on $M$, and by $\gS\gC(M,\cald)$ the subset of all strict contact structures $\eta$ such that $\eta=\ker\cald$. This subset splits non-canonically into two connected components $\gS\gC(M,\cald)^{\pm}$ given by fixing $\eta_o\in\gS\gC(M,\cald)$ and defining 
$$\gS\gC(M,\cald)^{\pm}= \{f\eta_o~|~f\in C^{\infty}(M),~ \begin{cases} ~\text{if $f>0$ ~everywhere}; \\
                                                                                                          ~ \text{if $f<0$ ~everywhere}.
                                                                                     \end{cases}$$
The map $\eta_o\mapsto -\eta_o$ reverses orientation when $n$ is even, but preserves the orientation when $n$ is odd. A choice of component defines a {\it co-orientation} on $(M,\cald)$. We assume in what follows that a contact structure $\cald$ on $M$ is both oriented and co-oriented. 

Given a strict contact structure $\eta\in \gS\gC(M,\cald)$, it is well known that there exists a unique vector field $R$ on $M$, called the {\it Reeb vector field}, that satisfies $\eta(R)=1$ and $R\hook d\eta=0$. On the $2n$-dimensional vector bundle $\cald$ we can choose a complex structure $J$ which gives $M$ a strictly pseudo-convex almost CR structure, denoted by $(\cald,J)$. Here we assume that $J$ is integrable, so $(\cald,J)$ is a CR structure on $M$. We can extend $J$ to a section $\Phi$ of the endomorphism bundle ${\rm End}(M)$ by setting 
\begin{equation}\label{JPhi}
\Phi R=0,\qquad \Phi |_\cald=J.
\end{equation}
We can now choose a compatible Riemannian metric $g$ on $M$ by requiring  
\begin{equation}\label{compatmetric}
g(\Phi X,\Phi Y)= g(X,Y) -\eta(X)\eta(Y)
\end{equation} 
holds for all vector fields $X,Y$. Then the quadruple $(R,\eta,\Phi,g)$ is called a {\it contact metric structure} on $M$. We make note of the relation $\Phi\circ\Phi=-\BOne +R\otimes \eta$ and that we can write $g$ as
$$g=d\eta\circ (\Phi\otimes \BOne) + \eta\otimes\eta.$$
We can extend the co-orientation involution $\gri:\gS\gC(M,\cald)^{+}\ra{1.8} \gS\gC(M,\cald)^{-}$ to the space of contact metric structures by sending 
the contact metric structure $\cals=(R,\eta,\Phi,g)$ to its {\it conjugate structure} $\cals^c=(-R,-\eta,-\Phi,g)$. We have

\begin{question}\label{codiff}
When is the involution $\gri$ induced by a diffeomorphism of $M$?
\end{question} 

Little appears to be known about this question at this stage, cf. Exercise 7.2 of \cite{BG05}.

\begin{definition}\label{Sasdef}
A {\it Sasakian structure} $\cals$ is a contact metric structure such that $(\cald,J)$ is an integrable CR structure, and the Reeb field $R$ is an infinitesimal isometry. In fact the transverse metric $g^T=d\eta\circ (\Phi\otimes \BOne)$ implies that the CR structure is strictly pseudoconvex. If $J$ is not necessarily integrable but $R$ is still an infinitesimal isometry, the quadruple $\cals$ is called a {\it K-contact} structure.  
\end{definition}

Note that if $\cals$ is Sasaki (K-contact) so is $\cals^c$. Recently, as discussed in the talk of A. Tralle, there has been much success in finding examples of K-contact manifolds that admit no Sasakian structure \cite{HaTr14,CNMY14,CaNiYu15,MuTr15,BFMT16,CNMY16}. This involves formality which also provides obstructions in the 3-Sasakian case \cite{FIM15} as discussed in the talk of M. Fern\'andez. 3-Sasakian structures are automatically Sasaki-Einstein and there are deformations of 3-Sasakian manifolds to Sasaki-Einstein structures that are not 3-Sasakian \cite{vCov13,vCov17}.  Here I concentrate mainly on the Sasaki case.

\begin{definition}\label{consas}
A contact structure $\cald$ on $M$ is said to be of Sasaki type if there exists a Sasakian structure $\cals=(R,\eta,\Phi,g)$ such that $\cald=\ker\eta$.
\end{definition}

In particular, if $\cald$ is of Sasaki type, then $\cald$ admits a compatible complex structure $J$ such that $(\cald,J)$ gives $M$ a strictly pseudo-convex CR structure.

For a general contact metric structure $\cals$ the dynamics of the Reeb vector field $R$ can be quite complicated; however, if $\cals$ is K-contact it is well understood. In particular, for K-contact structures the foliation $\calf_R$ is a Riemannian foliation. The structure $\cals$ or the foliation $\calf_R$ is said to be {\it quasi-regular} if there is a positive integer $k$ such that each point of $M$ has a foliated coordinate neighborhood $U$ such that each leaf of $\calf_R$ passes through $U$ at most $k$ times. If $k=1$ $\calf_R$ is called {\it regular}, and if there is no such $k$ it is called {\it irregular}. Quasi-regularity implies K-contact in which case the Reeb field $R$ generates a locally free $S^1$ action of isometries. For irregular K-contact structures the foliation $\calf_R$ generates an irrational flow on a torus in the automorphism group $\gA\gu\gt(\cals)$. The closure of the foliation $\calf_R$ gives a singular foliation $\bar{\calf}_R$.

\subsection{The Transverse Holomorphic Structure}
The transverse geometry of the Reeb foliation $\calf_R$ generated by $R$ is of great interest. In particular, we have \cite{BG05}[Proposition 6.4.8]
\begin{proposition}\label{RieK}
A contact metric structure $\cals=(R,\eta,\Phi,g)$ is K-contact if and only if the foliation $\calf_R$ is Riemannian, that is admits a transverse Riemannian structure.
\end{proposition}

In the Sasaki case the foliation $\calf_R$ is K\"ahlerian (i.e. has a transverse K\"ahler structure).
Let $\cals$ be a Sasakian (K-contact) structure on $M.$
From the viewpoint the Reeb foliation $\calf_R$ we denote the complex structure on the quotient  $\nu(\calf_R)=TM/\calf_R$ by $\bar{J}$ and call it the transverse holomorphic structure which we also denote by the pair $(\nu(\calf_R),\bar{J})$. The triple $(\nu(\calf_R),\bar{J},d\eta)$ gives $\nu(\calf_R)$ a natural transverse K\"ahler structure. Indeed, this means that the Riemannian foliation $\calf_R$ of a Sasakian (K-contact) structure $(R,\eta,\Phi,g)$ has a transverse K\"ahler (almost K\"ahler) structure given by the non-degenerate basic $(1,1)$-form $d\eta$ which represents a non-trivial element $[d\eta]_B$ in the basic cohomology group $H^{1,1}_B(\calf_R)$. Another important element of $H^{1,1}_B(\calf_R)$ is the basic first Chern class $c_1(\calf_R)$ (see 7.5.17 of \cite{BG05} for the precise definition). $2\pi c_1(\calf_R)$ can be represented by the transverse Ricci form of any Sasakian structure $\cals$ whose Reeb vector field is $R$ and transverse complex structure is $\bar{J}$. We recall the {\it type} of a Sasakian structure \cite{BG05}. 

\begin{definition}\label{Sastype}
A Sasakian structure $\cals=(R,\eta,\Phi,g)$ is {\bf positive (negative)} if the basic first Chern class $c_1(\calf_R)$ is represented by a positive (negative) definite $(1,1)$-form. It is {\bf null} if $c_1(\calf_R)=0$, and {\bf indefinite} if $c_1(\calf_R)$ is otherwise. 
\end{definition}

When $\cals$ is quasiregular so there is a $S^1$ orbibundle whose quotient is a projective algebraic orbifold $(\calz,\grD)$ where $\calz$ is a projective algebraic variety with cyclic quotient singularities and $\grD$ is a branch divisor, $c_1(\calf_R)$ is the pullback of the orbifold first Chern class $c_1^{orb}(\calz,\grD)$. In particular, a quasiregular Sasakian structure $\cals$ is positive if and only if its quotient orbifold $(\calz,\grD)$ is log Fano.

There is an exact sequence 
\begin{equation}\label{Sasexactseq}
0\ra{2,8}\bbr\fract{\grd}{\ra{2.8}}H^{2}_B(\calf_R)\fract{\gri_*}{\ra{2.8}}H^2(M,\bbr)\fract{j_2}{\ra{2.8}}H^1_B(\calf_R)\ra{2.8}\cdots
\end{equation}
where $\grd(a)=a[d\eta]_B$, $\gri_*$ is the map induced by forgetting that a basic closed 2-form is basic, and $j_2$ is the map $R\hook$ composed with the isomorphism $H^1(M,\bbr)\approx H^1_B(\calf_R)$. In particular, we have $\gri_*c_1(\calf_R)=c_1(\cald)\in H^2(M,\bbz)\subset H^2(M,\bbr)$, and we note that $c_1(\cald)$ is a contact invariant.

We remark that a choice of contact metric structure structure $\cals_o$ not only chooses a co-orientation, but also chooses an isomorphism $(TM/\calf_{R_o},\bar{J})\ra{1.8} (\cald,J)$ of complex vector bundles on $M$. 

\subsection{The Affine Cone}\label{affcone}
For any contact manifold $M$ we consider the cone $C(M)=M\times \bbr^+$. Choosing a 1-form $\eta$ in the contact structure of $M$ we form an exact symplectic structure on $C(M)$ by defining $\gro=d(r^2\eta)$ where $r\in\bbr^+$. The pair $(C(M),\gro)$ is called the symplectization of the (strict) contact structure $(M,\eta)$. Using the Liouville vector field $\Psi=r\partial_r$ we define a natural almost complex structure $I$ on $C(M)$ by
$$ IX=\Phi X +\eta(X)\Psi,\qquad  I\Psi=-R$$
where $X$ is a vector field on $M$, and $R$ is understood to be lifted to $C(M)$. Without further ado we shall identify $M$ with $M\times\{r=1\}\subset C(M)$. By adding the cone point we obtain an affine variety $Y=C(M)\cup\{0\}$ which is invariant under a certain complex torus action as described in the next section.

We have the well known \cite{BG05,CoSz12}

\begin{theorem}\label{saskahcone}
The quadruple $\cals=(R,\eta,\Phi,g)$ is Sasakian if and only if $(C(M),\gro,I,\bar{g})$ is K\"ahlerian with the cone metric $\bar{g}=dr^2+r^2g$. Moreover, in the Sasaki case the Sasakian structure $\cals$ corresponds to a polarized affine variety $(Y,I,R)$ polarized by the Reeb vector field $R$. Moreover, $(Y,I)$ is invariant under an effective action of a complex torus $\bbt^k_\bbc$ with $1\leq k\leq n+1$, and the Reeb vector field $R$ lies in the real Lie algebra $\gt_k$ of $\bbt^k$.
\end{theorem}

Such varieties have recently come under close scrutiny under the name of {\it $T$-varieties}, especially their close relationship with {\it polyhedral divisors}  \cite{AlHa06,AIPSV12} . Of course the toric case ($k=n+1$) has been well studied \cite{Oda88,Ful93,CoLiSc11}; but also the complexity one ($k=n$)  case has proven to be quite accessible \cite{Tim97,AlHa06,AlHaSu08,LiSu13}.
Choosing a Reeb vector field on the variety $(Y,I)$ fixes the K\"ahler form $\gro$ and its metric $\bar{g}$. Of particular interest is the $\bbq$-Gorenstein case and its relation with the existence of Sasaki-Einstein metrics \cite{FOW06,CFO07,CoSz15,Suss18}.

\begin{remark}\label{affvarrem}
The variety $Y$ plays two important roles for us. First, under certain circumstances from $Y$ we can obtain a {\it filling} $W$ that distinguishes contact structures of Sasaki type on $M$. Second, the polarized variety $(Y,R)$ can be used to obtain stability theorems (in the sense of geometric invariant theory). This gives rise to obstructions to the existence of extremal and constant scalar curvature Sasaki metrics.
\end{remark}

\subsection{The Einstein Condition}
Our study of Sasaki geometry, which began in the early 90's, was originally motivated by the search for Einstein metrics on compact manifolds. In that time it was unknown whether Einstein metrics were scarce or plentiful on compact manifolds. Indeed we eventually discovered a plethora of examples with huge moduli spaces, cf. \cite{BG05} and references therein.

Of course, the condition for a Riemannian metric $g$ to be Einstein is that its Ricci curvature $\Ric_g$ is proportional to the metric $g$. Since when $g$ is Sasaki we have the relation $\Ric_g(X,R)=2n\eta(X)$, any Sasaki-Einstein (SE) metric satisfies 
\begin{equation}\label{SEeqn}
\Ric_g=2ng.
\end{equation}
It follows that any SE structure is positive. This condition implies an equation in the basic cohomology group $H^2(\calf_R)$, namely,
\begin{equation}\label{cohSE}
c_1(\calf_R)=a[d\eta]_B,\qquad  a\in\bbr.
\end{equation}
From the exact sequence \eqref{Sasexactseq} this implies that $c_1(\cald)=0$ as a real cohomology class, or equivalently $c_1(\cald)$ is a torsion class. We are interested in finding solutions to Equation \eqref{cohSE} of the form 
\begin{equation}\label{etaEin}
\Ric_g=ag+(2n-a)\eta\otimes \eta, \qquad  a\in\bbr.
\end{equation}
Sasakian structures satisfying this equation are called {\it Sasaki-eta-Einstein}. Such metrics have constant scalar curvature.
The existence of Sasaki-eta-Einstein metrics comes from finding solutions to the Monge-Amp\`ere equation which in a local foliate coordinate chart takes the form
\begin{equation}
\frac{{\rm det}\Bigl(g_{i\bar{j}}+
\frac{\partial^2\phi}{\partial z_i\partial\bar z_j}\Bigr)} {{\rm
det}(g_{i\bar{j}})}=e^f\, 
\end{equation}
for some smooth function $f$.
It follows by the transverse Aubin-Yau Theorem of El Kacimi-Alaoui \cite{ElK} that such solutions exist in the negative ($a<-2$) and null ($a=-2$) cases. However, for positive Sasakian structures, as in the K\"ahler case, there are well known obstructions to the existence of such metrics. 

\subsection{Symmetries and Moment Maps}
There are various symmetry groups and algebras of interest to us \cite{Boy10a}. First on the level of contact structures we have the contactomorphism group $\gC\go\gn(M,\cald)=\{\phi\in \gD\gi\gf\gf(M)~|~ \phi_*\cald\subset \cald \}$. Choosing a contact 1-form $\eta$ representing $\cald$ chooses a co-orientation for $\cald$ and the component 
 of $\gC\go\gn(M,\cald)$ connected to the identity, viz.
\begin{equation}\label{contactgroup}
\gC\go\gn(M,\cald)^+= \{\phi\in \gD\gi\gf\gf(M)~|~ \phi^*\eta= e^f\eta \} 
\end{equation}
where $f\in C^\infty(M)$. For each such $\eta$ we have the closed subgroup $\gC\go\gn(M,\cald)^+$ of strict contact transformations
\begin{equation}\label{strcontactgroup}
\gC\go\gn(M,\eta)=\{\phi\in \gD\gi\gf\gf(M)~|~ \phi^*\eta =\eta\}.
\end{equation}
We denote the Lie algebras of $\gC\go\gn(M,\cald)^+$ and $\gC\go\gn(M,\eta)$ by  $\gc\go\gn(M,\cald)$ and $\gc\go\gn(M,\eta)$, respectively.

Fixing a strictly pseudo-convex CR structure $(\cald,J)$ we have the (almost) CR automorphism group 
$$\gC\gR(\cald,J)=\{\phi\in \gC\go\gn(M,\cald )~|~ \phi_*J=J\phi_*\}$$
and its subgroup $\gA\gu\gt(\cals)=\{\phi\in\gC\gR(\cald,J)~|~\phi_*R=R,~\phi^*g=g\}$ of K-contact or Sasakian automorphisms depending on whether $J$ is integrable or not. The group $\gA\gu\gt(\cals)$ is a compact Lie group and hence has a maximal torus $\bbt^k$ of dimension $1\leq k\leq n+1$ which is unique up to conjugacy in $\gA\gu\gt(\cals)$. 

From the point of view of the Reeb foliation $\calf_R$ we have the group $\gH(R,\bar{J})$ of transverse holomorphic transformations \cite{BGS06} defined by
\begin{equation*}\label{gphol}
\gH(R,\bar{J})=\{\phi\in \gF\go\gl(\calf_R)~|~\bar{\phi}\circ\bar{J}=\bar{J}\circ\bar{\phi} \}.
\end{equation*} 
where $\gF\go\gl(\calf_R)$ is subgroup of $\gD\gi\gf\gf(M)$ that leaves $\calf_R$ invariant, and $\bar{\phi}$ is the projection of the differential $\phi_*$ onto $\nu(\calf_R)$.
The group $\gH(R,\bar{J})$ is an infinite dimensional Fr\'echet Lie group, since the one parameter subgroup generated by any smooth section of the line bundle $L_R$ generated by $R$ is an element of $\gH(R,\bar{J})$. For this reason it is more convenient to use the approach in \cite{FOW06,CFO07,NiSe12}. In particular, following \cite{NiSe12} we define $\gA\gu\gt(\bar{J})$ of {\it group of transverse biholomorphisms} to be the group of complex automorphisms of $(C(M),I)$ that commute with $\Psi-iR$. This group desends to an action on $(M,\cals)$ commuting with $R$, and its Lie algebra by $\ga\gu\gt(\bar{J})$ coincides with the `Hamiltonian holomorphic vector fields' employed in \cite{FOW06,CFO07}. Note that  the connected component $\gA\gu\gt(\bar{J})_0$ of the group $\gA\gu\gt(\bar{J})$ contains the complexification $\bbt^k_\bbc$ of the maximal torus $\bbt^k$.

In the remainder of this section for simplicity we assume that our Sasaki manifold $M^{2n+1}$ satisfies $H^1(M,\bbr)=0$. We fix a Sasakian structure $\cals=(R,\eta,\Phi,g)$ on $M^{2n+1}$. There is a moment map $\mu:M\ra{1.8} \gc\go\gn(M,\eta)^*$ with respect to the Fr\'echet Lie group $\gC\go\gn(M,\eta)$ defined by
\begin{equation}\label{etamommap}
\langle\mu(x),X\rangle=\eta(X).
\end{equation} 
Since $X\in \gc\go\gn(M,\eta)$ the function $\eta(X)$ is basic, and $C^\infty_B$ has Lie algebra structure defined by the Jacobi-Poisson bracket defined by
\begin{equation}\label{JP}
\{f,g\}=\eta([{\rm grad~f},{\rm grad~g}]). 
\end{equation}
The evaluation map $X\mapsto \eta(X)$ is a Lie algebra isomorphism 
$$\gc\go\gn(M,\eta)\ra{2.8} C^\infty_B$$ 
and it follows, from the fact that the only element of $\gc\go\gn(M,\cald)$ that is a section of $\cald$ is the zero vector field, that the subalgebra $\bbr R$ is an ideal in the center of $\gc\go\gn(M,\eta)$. One then gets an isomorphism of the quotient algebras $\gc\go\gn(M,\cald)/\bbr R\approx C^\infty_B/\bbr$ where $\bbr$ denotes the ideal of constant functions. As in \cite{BGS06} we define $\calh^\cals_B$ as the subspace of $C^\infty_B$ that are solutions of the fourth-order differential equation
\begin{equation}\label{ord4}
(\bar{\partial}{\partial}^{\#})^{\ast}\bar{\partial}{\partial}^{\#}\varphi=0\, ,
\end{equation}
where ${\partial}^{\#}\varphi$ is the $(1,0)$ component of the gradient vector field, i.e $g({\partial}^{\#}\varphi,\cdot)=\bar{\partial}\varphi$. This gives a Lie algebra isomorphism 
\begin{equation}\label{transiso}
{\partial}^{\#}:\calh^\cals_B/\bbr\ra{1.6} \gh^T(R,\bar{J})/\grG(L_R).
\end{equation}

The Lie algebra $\gh^T(R,\bar{J})/\grG(L_R)$ contains the Lie algebra of a $k-1$ dimensional torus subgroup $\bbt^{k-1}$ of $\gC\go\gn(M,\eta)$ with $1\leq k\leq n+1$. This together with the group generated by the Reeb vector field gives a $k$ dimensional maximal torus $\bbt^k$ acting effectively on $M$.
We then have a $\bbt$-equivariant moment map $\mu:M\ra{1.8} \gt^*$ defined by
\begin{equation}\label{momentmap}
\langle\mu(x),\grz\rangle=\eta(X^\grz)
\end{equation}
where $X^\grz$ if the vector field corresponding to $\grz\in\gt$ under the action of $\bbt$. We define the $k$ dimensional Abelian subalgebra $\calh_B^\bbt \subset C^\infty_B$ of Killing potentials by
\begin{equation}
\calh_B^\bbt=\{\eta(X^\grz)~|~\grz\in\gt\}.
\end{equation}
which is a $k$ dimensional maximal Abelian subalgebra $\gt$ of the Lie algebra $C^\infty(M)_B$ of smooth basic functions on $M$. 
Note that $\calh_B^\bbt$ contains the constants $\bbr$ defined by $\eta(aR)=a$ corresponding to positive multiples of the Reeb vector field $R$. 

\begin{remark}
The symmetries of a contact structure of Sasaki type discussed here have lifts to symmetries of the affine variety $(Y,R)$ which we make use of from time to time.
\end{remark}

\subsection{The Sasaki Cone}
The space of Sasakian structures belonging to a fixed strictly pseudoconvex CR structure $(\cald,J)$ has an important subspace called the {\it unreduced Sasaki cone} and defined as follows. Fix a maximal torus $T$ in the CR automorphism group $\gC\gr(\cald,J)$ of a Sasaki manifold $(M,\eta)$  and let $\gt_k(\cald,J)$ denote the Lie algebra of $T$ where $k$ is the dimension of $T$. Then the unreduced Sasaki cone is defined by 
\begin{equation}\label{unsascone}
\gt^+_k(\cald,J)=\{R'\in\gt_k(\cald,J)~|~\eta(R')>0\}.
\end{equation}
It is easy to see that $\gt^+_k(\cald,J)$ is a convex cone in $\gt_k=\bbr^k$. It is also well known that $1\leq k\leq n+1$ with $k=n+1$ being the {\it toric} case.
Then the {\it reduced Sasaki cone} is defined by $\grk(\cald,J)=\gt^+_k(\cald,J)/\calw(\cald,J)$ where $\calw(\cald,J)$ is the Weyl group of $\gC\gr(\cald,J)$. The reduced Sasaki cone $\grk(\cald,J)$ can be thought of as the moduli space of Sasakian structures whose underlying CR structure is $(\cald,J)$. We shall often suppress the CR notation $(\cald,J)$ when it is understood from the context. We also refer to a Sasakian structure $\cals$ as an element of $\gt^+_k(\cald,J)$. We view the Sasaki cone $\gt^+_k(\cald,J)$ as a $k$-dimensional smooth family of Sasakian structures. It is often convenient to use the {\it complexity} $n+1-k$ in lieu of the $\dim\gt^+_k(\cald,J)=k$. Then the complexity is an integer from $0$ to $n$ with $0$ being the toric case.

We can lift the Sasaki cone to the affine variety $Y$ \cite{CoSz12} in which case the torus action gives the ring $\gH$ of holomorphic functions on $Y$ the weight space decomposition 
$$\gH =\bigoplus_{\alpha\in\calw} \gH_{\alpha}$$ 
where $\calw\subset\gt^*_k$ is the set of {\it weights}. In this case the Sasaki cone takes the form
\begin{equation}\label{SasconeY}
\gt^+_k =\{ R\in\gt\ |\ \forall~ \alpha\in\calw, \alpha\neq 0 \text{ we have } \alpha(R)>0 \}.
\end{equation}

\subsection{Positive Sasakian Structures}
The type of Sasaki structure as defined in Definition \ref{Sastype} is an important invariant. It is well known \cite{BG05} that if $\dim \gt^+>1$ then $\cals$ must be of positive or indefinite type. We let $\gp^+$ denote the subset of positive Sasakian structures in the Sasaki cone $\gt^+$. The subset $\gp^+$ is conical in the sense that if $\cals\in \gp^+$ so is $\cals_a$; however, it is still open whether $\gp^+$ is connected generally. Positive Sasakian structures are important since they give rise to Sasaki metrics with positive Ricci curvature \cite{ElK,BG05}. Moreover, it turns out that the space $\gp^+$ is related to the total transverse scalar curvature $\bfS_R$ defined in Equation \eqref{VS} below. In fact we have \cite{BHL17} 
\begin{proposition}\label{SRpos}
Let $(M,\cals)$ be a Sasaki manifold with $\dim \gt^+>1$. Suppose also that the total transverse scalar curvature $\bfS_R$ of $\cals$ is non-positive. Then $\gp^+ =\emptyset$.
\end{proposition}

Equivalently, $\gp^+ =\emptyset$ means that all Sasakian structures in $\gt^+$ are indefinite. See also Theorem \ref{BHLthm} below. The other extreme is that $\gp^+=\gt^+$, and this occurs when $c_1(\cald)$ can be represented by positive definite $(1,1)$ form which we denote by $c_1(\cald)>0$. It also occurs when $c_1(\cald)=0$ and $\gt^+$ contains any positive Sasakian structure. This is called the monotone case. This gives rise to the important 

\begin{question}\label{p+ques}
For which contact manifolds of Sasaki type does $\gt^+=\gp^+$?
\end{question}

Clearly, aside from those mentioned above, when $\dim\gt^+=1$ if there exists one positive Sasakian structure, then $\gt^+=\gp^+$. In \cite{BoTo15} the authors described an invariant that answers this question in the case of certain $S^3_\bfw$ joins (cf. Section \ref{joinsect} below). This invariant depends on a primitive class $\grg$ in $H^2(M,\bbz)$. Given a non-zero first Chern class $c_1(\cald)\in H^2(M,\bbz)$ assume that there exists a non-negative real number $B'$  and a primitive class $\grg\in H^2(M,\bbz)$ such that $c_1(\cald)+B'\grg>0$. We define $B$ by
\begin{equation}\label{supB}
B_\grg=\inf\{B'\in\bbr_{\geq 0} ~|~ c_1(\cald)+B'\grg>0\}.
\end{equation} 
Note that if $H^2(M,\bbz)=0$ there is no such invariant.
When $H^2(M,\bbz)\neq 0$, $\gp^+\neq\emptyset$ and $\dim\gt^+>1$, but still no such $B_\grg$ exists, we can have the phenomenon of {\it type changing} in $\gt^+$ \cite{BoTo15} in which case $\gp^+$ is a proper subset of $\gt^+$.

\section{Constructions of Sasaki Manifolds}\label{Consect}
In this section we describe methods for the explicit construction of Sasaki manifolds. These are divided up into 5 different constructions. The first is general, up to deformations, and is used throughout. The second construction is used at times, but is not fully developed in this paper. We concentrate here on the third and fourth constructions, while the fifth is currently a work in progress \cite{BoTo18b} and will not be described here.
\begin{enumerate}
\item the total space $M$ of an $S^1$-orbibundle over a projective algebraic orbifold;
\item Sasakian manifolds with many symmetries, e.g. toric contact structures of Reeb type, or more generally those induced by $T$-varieties;
\item links of weighted homogeneous polynomials, e.g. Brieskorn manifolds;
\item the Sasaki join construction;
\item Yamazaki's fiber join construction \cite{Yam99}. 
\end{enumerate}

Item (1) is foundational and general (up to deformations), whereas item (2)-(5) are quite specialized.  In this paper we make use of items (3) and (4), so we describe these below in some detail. First we shall often use item (1) in the form:
\begin{theorem}\label{fundthm2} 
Let $(\calz,h)$ be a compact Hodge orbifold with K\"ahler form $\gro$.
Let $\pi:M\ra{1.3} \calz$ be the $S^1$ orbibundle whose first Chern
class is $[\gro],$ and let $\eta$ be a connection 1-form in $M$ whose
curvature is $d\eta=\pi^*\gro,$ then $M$ with the metric
$\pi^*h+\eta\otimes\eta$ is a Sasaki orbifold.  Furthermore, if all the local uniformizing groups inject into the group of the bundle $S^1,$ the total space $M$ is a smooth Sasaki manifold. Furthermore, all quasi-regular Sasaki manifolds are obtained this way.
\end{theorem}

The fiducial example of a Sasaki manifold is the standard round sphere $S^{2n+1}$ constructed as the total space of the Hopf fibration over the complex projective space $\bbc\bbp^n$. This example is regular. We obtain a weighted quasi-regular version whose quotient is a weighed projective space $S^{2n+1}\ra{1.6} \bbc\bbp[w_0,\ldots,w_n]$ using the weighted $S^1$ action on the sphere induced by \eqref{c*act} below. We refer to the sphere with this action as the weighted sphere, denoted by $S^{2n+1}_\bfw$.

\subsection{Hypersurfaces of Weighted Homogeneous Polynomials}\label{Briessect} 
These are constructed as the links of isolated hypersurface singularities of weighted homogeneous polynomials, in particular links of Brieskorn manifolds and their deformations. We define a weighted $\bbc^*$ action on $\bbc^{n+1}$ by 
\begin{equation}\label{c*act}
\bfz=(z_0,\cdots,z_n)\mapsto (\grl^{w_0}z_0,\cdots,\grl^{w_n}z_n),
\end{equation}
where $\grl\in\bbc^*$ and which we denote by $\bbc^*(\bfw)$ where $\bfw=(w_0,\cdots,w_n)\in(\bbz^+)^{n+1}$ is the {\it weight vector} and the monomial $z_i$ is said to have weight $w_i$. This induces an action of $\bbc^*(\bfw)$ on the space of polynomials $\bbc[z_0,\ldots,z_n]$ whose eigenspaces are {\it weighted homogeneous polynomials} of {\it degree} $d$ defined by 
\begin{equation}\label{whp}
f(\lambda^{w_0} z_0,\ldots,\lambda^{w_n}
z_n)=\lambda^df(z_0,\ldots,z_n)\, .
\end{equation}
We want to make two assumptions about $f$. First, $f$ should have only an isolated singularity at the origin in $\bbc^{n+1}$, and second, $f$ does not contain monomials of the form $z_i$ for any $i=0,\dots,n$. This last condition eliminates the standard constant curvature metric sphere. Here we are concerned with Brieskorn-Pham polynomials of the form 
\begin{equation}\label{BPpoly}
f(\bfz)=z_0^{a_0}+\cdots +z_n^{a_n}
\end{equation}
and their perturbations obtained by adding to $f$ monomials of the form $z_0^{b_0}\cdots z_n^{b_n}$ where $0\leq b_j<a_j$ and $\sum_jb_jw_j=d$. Here the exponent vector $\bfa=(a_0,a_1,\ldots,a_n)\in (\bbz^+)^{n+1}$ is related to the weights by $a_iw_i=d$ for all $i=0\cdots,n$. The link 
\begin{equation}\label{Brieslink}
L(\bfa)=\{f(\bfz)=0\}\cap S^{2n+1}
\end{equation}
associated to the polynomial \eqref{BPpoly} is called a {\it Brieskorn manifold}. The link obtained from perturbations of a Brieskorn manifold obtained by adding a polynomial $p(z_0,\ldots,z_n)$ is denoted by $L(\bfa,p)$. Both $L(\bfa)$ and $L(\bfa,p)$ have natural contact structures of Sasaki type. We denote the affine cone with associated $L(\bfa)$ and $L(\bfa,p)$ by $Y(\bfa)$ and $Y(\bfa,p)$, respectively. The condition, which we call GC \cite{BGK05}, that is needed on the polynomial $p$ is that the intersections of $Y(\bfa,p)$ with any number of hyperplanes $z_i=0$ are smooth outside the origin. The following is essentially due to Takahashi \cite{Tak78}

\begin{theorem}\label{Briessasthm}
The links $L(\bfa)$ and $L(\bfa,p)$ have a natural Sasakian structure which is the restriction of the weighted Sasakian structure $\cals_\bfw=(R_\bfw,\eta_\bfw,\Phi_\bfw,g_\bfw)$ on $S^{2n+1}_\bfw$. 
\end{theorem}

Such a theorem and similar results also hold for links of complete intersections of weighted homogeneous polyomials which for simplicity we do not discuss. In any case these all are $\bbq$-Gorenstein in the sense that $c_1(\cald)$ vanishes in de Rham cohomology. Again for simplicity we restrict our attention to the simply connected case.
The existence proofs for Sasaki-Einstein metrics are treated in detail in our book \cite{BG05} and as well appear in our survey of over 10 years ago \cite{Boy08} we do not discuss further here.

\subsection{The $S^3_\bfw$ Join Construction}\label{joinsect} 
The join construction is the Sasaki analogue for products in K\"ahler geometry. It was first described in the context of Sasaki-Einstein manifolds \cite{BG00a}, but then developed more generally in \cite{BGO06}, see also Section 7.6.2. of \cite{BG05}. Its relation to the de Rham decomposition Theorem and reducibility questions \cite{HeSu12b} are studied in \cite{BHLT16}. We consider the set $\cals\calo(\cals\calm)$ of all Sasaki orbifolds (manifolds), respectively. $\cals\calo$ is the object set of a groupoid whose morphisms are orbifold diffeomorphisms. Moreover, $\cals\calo$ is graded by dimension and has an additional multiplicative structure described in Section 2.4 of \cite{BHLT16} which we call the {\it Sasaki $l$-monoid}. The positive irreducible generators in dimension 3 are of the form $\star_{l_1,l_2}S^3_\bfw$ where the weight vector $\bfw=(w_1,w_2)$ has relatively prime components $w_i\in \bbz^+$ that are ordered as $w_1\geq w_2$. We apply this to the submodule $\cals\calm$ which is not an ideal in general. However, $\star_{l_1,l_2}S^3_\bfw$ does give a map $\cals\calm\ra{1.8}\cals\calm$ if the condition
\begin{equation}\label{adcond}
\gcd(l_2,l_1w_1w_2)=1) 
\end{equation}
holds in which case we say that $(l_1,l_2,\bfw)$ are {\it admissible}. In this case the join is then constructed from the following  commutative diagram
\begin{equation}\label{s2comdia}
\begin{matrix}  M\times S^3_\bfw &&& \\
                          &\searrow\pi_L && \\
                          \decdnar{\pi_{2}} && M_{\bfl,\bfw} &\\
                          &\swarrow\pi_1 && \\
                         N\times\bbc\bbp^1[\bfw] &&& 
\end{matrix} 
\end{equation}
where the $\pi_2$ is the product of the projections of the standard Sasakian projections $\pi_M:M\ra{1.6}N$ and $S^3_\bfw\ra{1.6} \bbc\bbp^1[\bfw]$ given by Theorem \ref{fundthm2}. The circle projection $\pi_L$ is generated by the vector field 
\begin{equation}\label{Lvec}
L_{\bfl,\bfw}=\frac{1}{2l_1}R_1-\frac{1}{2l_2}R_\bfw,
\end{equation}
and its quotient manifold $M_{\bfl,\bfw}$, which is called the $(l_1,l_2)$-join of $M$ and $S^3_\bfw$ and denoted by $M_{\bfl,\bfw}=M\star_{l_1,l_2}
S^3_\bfw$, has a naturally induced quasi-regular Sasakian structure $\cals_{\bfl,\bfw}$ with contact 1-form $\eta_{l_1,l_2,\bfw}$. It is reducible in the sense that both the transverse metric $g^T$ and the contact bundle $\cald_{l_1,l_2,\bfw}=\ker\eta_{l_1,l_2,\bfw}$ split as direct sums. Since the CR-structure $(\cald_{l_1,l_2,\bfw},J)$ is the horizontal lift of  the complex structure on $N\times\bbc\bbp^1[\bfw]$, this splits as well. The choice of $\bfw$ determines the transverse complex structure $J$. This construction provides a family of contact structures of Sasaki type on a family of smooth manifolds whose cohomology ring can, in principle, be computed.  The Sasakian structures that are most accessible through this construction are the elements of the 2-dimensional subcone $\gt^+_\bfw$ of the Sasaki cone $\gt^+$ known as the $\bfw$ subcone. If $N$ has no Hamiltonian symmetries then $\gt^+=\gt^+_\bfw$. Note that $\star_{l_1,l_2}S^3_\bfw$ induces addition of the Lie algebras $\gt_M$ and $\gt_\bfw$, and hence addition of the Sasaki cones, \
\begin{equation}\label{t+add}
(\gt^+_M,\gt^+_\bfw)\mapsto \gt^+_M+\gt^+_\bfw.
\end{equation}

An important property of the $S^3_\bfw$ join construction is that one can apply the admissible construction of Apostolov, Calderbank, Gauduchon, and T{\o}nnesen-Friedman \cite{ApCaGa06,ACGT04,ACGT08} to give explicit constructions of existence theorems for extremal and constant scalar curvature Sasakian structures. Since this was presented in a recent survey \cite{BoTo14P} as well ss our original papers \cite{BoTo13,BoTo14a}, we do not give a description of this construction here.

\section{Deformation Classes and Moduli}
We denote by ${\mathcal S}(M)$ the space of all Sasakian structures on $M$ and give it the $C^\infty$ Fr\'echet topology as sections of vector bundles, and denote by ${\mathcal S}(M,R,\bar{J})$ the subset of Sasakian structures with Reeb vector field $R$, with transverse holomorphic structure $\bar{J}$, and the same complex structure on the cone $C(M)$. We give ${\mathcal S}(M,R,\bar{J})$ the subspace topology. A Sasakian structure chooses a preferred basic cohomology class $[d\eta]_B\in H^{1,1}_B(\calf_R)$, and if $\cals'=(R,\eta',\Phi',g')$ is another Sasakian structure in ${\mathcal S}(M,R,\bar{J})$, then we have $[d\eta']_B=[d\eta]_B$. Using the transverse $\partial\bar{\partial}$ lemma \cite{ElK} we can identify ${\mathcal S}(M,R,\bar{J})$ with the contractible space 
\begin{equation}\label{phidef}
\{\phi\in C^\infty_B~|~(\eta+d_B^c\phi)\wedge (d\eta +i\partial_B\bar{\partial}_B\phi)^n\neq 0,~\int_M\phi~\eta\wedge (d\eta)^n=0\},
\end{equation}
where $d_B^c=\frac{i}{2}(\bar{\partial}-\partial).$ For each such $\phi$ we have a Sasakian structure $\cals_\phi=(R,\eta_\phi,\Phi_\phi,g_\phi)$ with $\eta_\phi=\eta +d_B^c\phi,\Phi_\phi=\Phi-R\otimes d_B^c\phi,$ and $g_\phi=d\eta_\phi\circ (\Phi_\phi\otimes \BOne)+\eta_\phi\otimes\eta_\phi$ with the same Reeb vector field $R$, the same transverse holomorphic structure $\bar{J}$ and the same holomorphic structure on the cone. The class $[d\eta]_B\in H^{1,1}_B(\calf_R)$ is called a {\it transverse K\"ahler class} or {\it Sasaki class}. Here the contact structure $\cald$ changes but in a continuous way giving an equivalent contact structure by Gray's Theorem \cite{Gra59}. The {\it isotopy class} $\bar{\cald}$ of contact structures is unchanged. This type of deformation is used to find extremal and constant scalar curvature Sasakian structures. The space ${\mathcal S}(M,R,\bar{J})$ described above is clearly infinite dimensional, so we want to factor this part out to get a finite dimensional moduli space. So we consider the identification space ${\mathcal S}(M)/{\mathcal S}(M,R,\bar{J})$. It is a pre-moduli space of Sasaki classes which we denote by $\gP\gM^{cl}_\cals$. The Sasaki class of a Sasakian structure $\cals=(R,\eta,\Phi,g)$ is denoted by $\bar{\cals}$. The diffeomorphism group $\dif$ acts on ${\mathcal S}(M)$ by sending $\cals=(R,\eta,\Phi,g)$ to $\cals^\varphi=(\varphi_*R,(\varphi^{-1})^*\eta,\varphi^{-1}_*\Phi\varphi_*,(\varphi^{-1})^*g)$ for $\varphi\in \dif$. Then the transformed structure $\cals_\phi^\varphi$ belongs to the transformed space ${\mathcal S}(M,\varphi_*R,\varphi^{-1}_*\bar{J}\varphi_*)$. So $\dif$ acts on $\gP\gM^{cl}_\cals$ to define the {\it moduli space of Sasaki isotopy classes} $\gM^{cl}_\cals(M)$. There is another equivalence of importance to us, namely, {\it transverse scaling} or {\it transverse homothety}. It is defined by sending $\cals=(R,\eta,\Phi,g)$ to $\cals_a=(a^{-1}R,a\eta,\Phi,g_a)$ for $a\in\bbr^+$ and $g_a=a(g+(a-1)\eta\otimes\eta)$. One easily checks that if $\cals$ is Sasaki (K-contact) then so is $\cals_a$. Thus, $\gM^{cl}_\cals(M)$ is a conical space which is generally not connected, that is, for each element $\bar{\cals}\in\gM^{cl}_\cals(M)$ one has the {\it ray} $\gr_\cals=\{\cals_a~|~a\in\bbr^+\}$. In the sequel we shall study local conical subspaces namely, the {\it Sasaki cones} and {\it Sasaki bouquets} described in Section \ref{Sasconesect}.

Generally very little is known about the space $\gM^{cl}_\cals(M)$. In particular, $\pi_0(\gM^{cl}_\cals(M))$ is of much interest. But $\gM^{cl}_\cals(M)$ can be non-Hausdorff. Locally it is determined by the deformation theory of the transverse holomorphic structure of the foliation $\calf_R$. So we fix the contact structure and deform the transverse holomorphic structure using Kodaira-Spencer theory. However, as shown in \cite{Noz14} deforming the transversely holomorphic foliation of a Sasakian structure may not stay Sasakian. The obstruction for remaining Sasakian is the $(0,2)$ component of the basic Euler class in the basic cohomology group $H^{2}_B(\calf_R)$. Fortunately, for those that we are most interested in, namely positive Sasakian structures there is a vanishing theorem\footnote{This vanishing theorem was first proved for quasi-regular Sasakian structures in \cite{BGN03a} and recently proved in full generality by Nozawa \cite{Noz14}.} $H^{0,q}_B(\calf_R)=0$. 

For any contact structure $\cald$ the first Chern class $c_1(\cald)$ is classical invariant. We have

\begin{proposition}\label{c1inv}
Let $\cald$ and $\cald'$ be two contact structures on $M$ such that $c_1(\cald)=m\grg$ and $c_1(\cald')=m'\grg$ for some $m,m'\in\bbz$ and some primitive class $\grg\in H^2(M,\bbz)$. Then if $m'\neq m$, the contact structures $\cald$ and $\cald'$ are inequivalent. 
\end{proposition}

For a fixed contact structure $\cald_\cals$ of Sasaki type we denote by $\gM^{cl}_\grg(M)$ the subspace of $\gM^{cl}_\cals(M)$ such that $c_1(\cald)=\grg\in H^2(M,\bbz)$.
We also want to distinguish Sasakian structures by their type per Definition \ref{Sastype}, so we define the corresponding subspaces $\gM^{cl}_{+}(M),\gM^{cl}_{-}(M),\gM^{cl}_{N}(M)$ of, respectively positive, negative, null Sasakian structures and $\gM^{cl}_{I}(M)$ the moduli of indefinite Sasaki classes. The exact sequence \eqref{Sasexactseq} gives relations between these subspaces. It is easy to see that \cite{BoTo15}

\begin{proposition}\label{typegamma}
The following hold:
\begin{enumerate}
\item $\gM^{cl}_{0}(M)\subset \gM^{cl}_{+}(M)\cup \gM^{cl}_{-}(M)\cup \gM^{cl}_{N}(M)$;
\item If $\grg$ is represented by a positive definite $(1,1)$-form, then $\gM^{cl}_{\grg}(M)\subset \gM^{cl}_{+}(M)$.
\item If $\grg$ is represented by a negative definite $(1,1)$-form, then $\gM^{cl}_{\grg}(M)\subset \gM^{cl}_{-}(M)$ and $\dim\gt^+=1$.
\end{enumerate}
\end{proposition}


We denote the {\it moduli space of positive (negative) Sasaki classes with $c_1(\cald)=\grg$} by $\gM^{cl}_{\pm,\grg}(M)$ where $\grg\in H^2(M,\bbz)$. Generally, this space is not connected, so $\pi_0(\gM^{cl}_{\pm,\grg}(M))$ is an important invariant as we shall see. For null Sasakian structures we must have $c_1(\cald)=0$, so the moduli space is denoted by $\gM^{cl}_{N}$.
We are also interested in the question of whether a class in $\gM^{cl}_{\pm,\grg} (M)$ has an extremal representative or not as described in Section \ref{extsassec} below. 
Within the class of extremal Sasakian structures perhaps the most important representatives are those with constant scalar curvature (CSC). Accordingly, we denote the moduli space of Sasakian structures with constant scalar curvature by $\gM^{CSC}_{\grg} (M)$. Note that by \cite{ElK} when $c_1(\cald)=0$ the negative and null Sasakian structures in $\gM^{cl}_0(M)$ always admit a CSC Sasaki metric. These are all the negative and null Sasaki-eta-Einstein metrics. In particular, for null Sasakian structures every class in $\gM^{cl}_{N}$ has a CSC representative.

We shall be interested in local moduli spaces of Sasaki-Einstein metrics and Sasaki-eta-Einstein metrics. To this end we let $\cals\cale(M)$ denote the subspace of Sasaki-Einstein metrics in ${\mathcal S}(M,R,\bar{J})$, and $\cals\eta\cale(M)$ the subspace of Sasaki-eta-Einstein metrics in ${\mathcal S}(M,R,\bar{J})$. We make note of the fact that when $c_1(\cald)=0$ a CSC Sasaki metric is automatically Sasaki-eta-Einstein. Then we define the {\it local moduli space of Sasaki-Einstein metrics} by $$\gM^{SE}(M)=\cals\cale(M)/\gA\gu\gt(\bar{J})_0$$
and the {\it local moduli space of Sasaki-eta-Einstein metrics} by 
$$\gM^{CSC}_{+,0}(M)=\cals\eta\cale(M)/\gA\gu\gt(\bar{J})_0.$$
A well known result of Tanno says that 
$\gM^{CSC}_{+,0}(M)\approx \gM^{SE}(M)\times\bbr^+$.

Since Sasaki-Einstein structures are positive, elements of $\gM^{SE}(M)$ are representatives of elements of $\gM^{cl}_{+,0}(M)$ and it is well known that not all classes in $\gM^{cl}_{+,0}(M)$ have a representative in $\gM^{SE}(M)$. There is a natural map $\gc:\gM^{SE}(M)\ra{1.6} \gM^{cl}_{+,0}(M)$ which sends an SE structure to its Sasaki class. By a theorem of Nitta and Sekiya \cite{NiSe12} if $\gc(\cals)=\gc(\cals')$ then there is a $g$ in the connected component $\gA\gu\gt(\bar{J})_0$ of the group of transverse holomorphic automorphisms such that $\cals'=g(\cals)$. Hence, there is an injective map $\gM^{SE}(M)\ra{1.8} \gM^{cl}_{+,0}(M)$.


By a theorem of Girbau, Haefliger, and Sundararaman (Theorem 8.2.2 in \cite{BG05}) when $H^2(M,\Theta_{\calf_{R_0}})=0$  the Kuranishi space of deformations is identified with an open neighborhood of $0$ in $H^1(M,\Theta_{\calf_{R}})$ where $\Theta_{\calf_R}$ is the sheaf of germs of transverse holomorphic vectors of a Sasakian structure $\cals=(R,\eta,\Phi,g)$. In this case assuming that $\cals=(R,\eta,\Phi,g)$ is quasiregular with quotient orbifold  $\calz$ Proposition 8.2.6 of \cite{BG05} gives an exact sequence
$$0\ra{2.5}H^1(\calz,\Theta_\calz)\ra{2.5} H^1(M,\Theta_{\calf_{R}})\ra{2.5} H^0(\calz,\Theta_\calz)\ra{2.5}H^2(\calz,\Theta_\calz)\ra{2.5}0$$
where $\Theta_\calz$ is the sheaf of germs of holomorphic vector fields of the orbifold $\calz$.
Here we consider two types of deformations, those given by the injection of the second arrow, and those coming from global sections of the toral subsheaf $\gT\subset \Theta_\calz$. The latter are deformations within the Sasaki cone described in Section \ref{Sasconesect}.
Later we describe in more detail deformations of the transverse complex structure in the special case of Brieskorn manifolds. More generally deformations of the transverse complex structure have also been studied in \cite{vCov15,Noz14}.

\begin{example}\label{standardsphere}
Let us consider the {\it standard} contact structure on an odd dimensional sphere $(S^{2n+1},\cald_o)$ which is ``standard'' in many ways. It is of Sasaki type and there is a complex structure $J_o$ on $\cald_o$ giving the {\it standard} CR structure $(\cald_o,J_o)$ on $S^{2n+1}$. Within this contact CR structure there is a 1-form $\eta_o$ whose corresponding Sasakian structure $\cals_o$ is Sasaki-Einstein and the metric $g_o$ is the {\it standard} Riemannian metric $g_o$ on $S^{2n+1}$ with sectional curvature $1$. Moreover, the $S^1$ action on $S^{2n+1}$ gives the well known Hopf fibration $S^{2n+1}\ra{1.8}\bbc\bbp^n$ which is a Riemannian submersion with the Fubini-Study metric on $\bbc\bbp^n$. It is well known that the connected component of the Sasaki automorphism group $\gA\gu\gt(\cals_o)$ is $U(n)\times S^1$, so there is a maximal torus $\bbt^{n+1}$ of dimension $n+1$ showing that the contact structure $\cald_o$ is toric.

Of course any deformation of $\cald_o$ of the form $\eta_o\mapsto \eta_\phi=\eta_o +d_B^c\phi,$  where $\phi$ satisfies \eqref{phidef} is contactomorphic to $\cald_o$ by Gray's Theorem \cite{Gra59}. So we shall often refer to the isotopy class $\bar{\cald}_o$ as the standard contact structure as well. Note that not all representatives of the isotopy class $\bar{\cald}_o$ are toric since $\phi$ is not necessarily invariant under $\bbt^{n+1}$. Nevertheless, we consider such structures to be toric since they are deformation equivalent to a toric structure. 

We know, however, that we can also deform the Reeb vector field $R_o$ giving a family of Sasakian structures whose underlying CR structure is $(\cald_o,J_o)$, namely, the Sasaki cone $\gt^+_{n+1}(\cald_o,J_o)$ of dimension $n+1$, and we can identify the moduli space $\gM^{cl}_o$ of standard Sasakian structures with the reduced Sasaki cone $\grk(\cald_o,J_o)$.


A contact structure on $S^{2n+1}$ that is not contactomorphic to the standard one $\cald_o$ is called an {\it exotic} contact structure on $S^{2n+1}$. As explained by Kwon and van Koert \cite{KwvKo13} the following proposition is a consequence of the work of Eliashberg, Gromov and McDuff, cf. \cite{Eli91}:
\begin{proposition}\label{Briesexotic}
Let $L(\bfa)$ be a Brieskorn manifold that is diffeomorphic to $S^{2n+1}$ such that all $a_i\geq 2$. Then the natural contact structure on $L(\bfa)$ is exotic.
\end{proposition}

\end{example}


\subsection{Deformations of the Transverse Complex Structure and the Sasaki Bouquet}\label{Sasconesect}
A deformation of the transverse complex structure induces a deformation of the CR structure which by Gray's Theorem preserves the contact structure $\cald$. 
For each strictly pseudoconvex CR structure $(\cald,J)$ there is a unique conjugacy class of maximal tori in $\gC\gr(\cald,J)$. This in turn defines a conjugacy class $\calc_T(\cald)$ of tori in the contactomorphism group $\gC\go\gn(\cald)$ which may or may not be maximal. We thus recall \cite{Boy10a} the map $\gQ$ that associates to any transverse almost complex structure $J$ that is compatible with the contact structure $\cald$, a conjugacy class of tori in $\gC\go\gn(\cald)$, namely the unique conjugacy class of maximal tori in $\gC\gR(\cald,J)\subset \gC\go\gn(\cald)$. Then two compatible transverse almost complex structures $J,J'$ are {\it T-equivalent} if $\gQ(J)=\gQ(J')$. A {\it Sasaki bouquet} is defined by
\begin{equation}\label{sasbouq}
\gB_{|\cala|}(\cald)=\bigcup_{\gra\in\cala}\gt^+(\cald,J_\gra)
\end{equation}
where the union is taken over one representative of each $T$-equivalence class in a preassigned subset $\cala$ of $T$-equivalence classes of transverse (almost) complex structures. Here $|\cala|$ denotes the cardinality of $\cala$. Since we are considering only deformations of the transverse complex structure, we restrict ourselves to the case that $J_\gra$ is integrable.

We can apply deformation theory to the two constructions of Section \ref{Consect}:
\begin{enumerate}
\item hypersurfaces of weighted homogeneous polynomials;
\item the $S^3_\bfw$ join construction.
\end{enumerate}

Deforming a weighted homogeneous polynomial $f$ of degree $d$ (only through weighted homogeneous polynomials) gives a smooth local moduli space $\gM^{whp}$ at $f$ satisfying
$$\dim_\bbc\gM^{whp}\geq h^0(\bbc\bbp(\bfw),d)-\sum_ih^0(\bbc\bbp(\bfw),w_i) +\dim\gA\gu\gt(L(\bfa))$$
where $h^0=\dim_\bbc H^0$ is the complex dimension of the corresponding space of sections. In many cases these give rise to a local moduli space $\gM^{SE}$ of Sasaki-Einstein metrics, cf. \cite{BGK05}. In this case $|\cala|=1$ so there is a $T$-equivalent family of Sasaki cones $\gt^+(\cald,J_t)$ in a small enough neighborhood of $f$ in $\gM^{whp}$ such that the Sasaki cones belong to inequivalent underlying CR structures. So in this case the bouquet consists of a single $T$-equivalence class of Sasaki cones, and we say that the bouquet is {\it trivial}. We emphasize that a trivial bouquet does not mean that the moduli space is trivial. On the other hand deforming the transverse complex structure of an $S^3_\bfw$ join manifold $M_{\bfl,\bfw}$ can give a non-Hausdorff local moduli space in the form of a non-trivial bouquet. We give examples of both of these deformations later in Section \ref{Exmodsect}. 

\subsection{Distinguishing Contact Structures of Sasaki Type}
A fundamental theorem due to John Gray, known as the Gray Stability Theorem, says that there are no local invariants, that is, all deformations of a contact structure $\cald$ are trivial. Thus, one looks for discrete invariants. Clearly, we have the classical invariant, namely, the first Chern class $c_1(\cald)$. Two methods for distinguishing contact structures with the same first Chern class is (1) the contact homology of Eliashberg, Giventhal, and Hofer when the transversality issues have been resolved; (2) the  $S^1$-equivariant symplectic homology  of an appropriate filling introduced by Viterbo \cite{Vit99} and developed further by Bourgeois and Oancea \cite{BoOa13a}. In this survey we concentrate mainly on (2). Here we give a brief review referring to \cite{BMvK15,BoOa13a} for details.

As mentioned previously a contact manifold of Sasaki type is holomorphically fillable; however, to compute our invariants we need a stronger condition on the filling. 
A {\it holomorphic filling} of a compact, coorientable contact manifold $(M,{\mathcal D})$ consists of a compact complex manifold with boundary $(W,J)$ such that
\begin{itemize}
\item the boundary of $W$ is diffeomorphic to $M$;
\item the boundary of $W$ is $J$-convex and $\mathcal D=TM \cap JTM$.
\end{itemize}
A {\it Stein filling} of a contact manifold $(M,{\mathcal D})$ is a holomorphic filling of $(M,{\mathcal D})$ that is biholomorphic to a Stein domain. As mentioned in the introduction

\begin{proposition}\label{Sasholfill}
A contact manifold of Sasaki type is holomorphically fillable.
\end{proposition}

In this case we can resolve the singularity of the affine variety $Y$ and cut off the cone at $r=1$ so that the boundary of the filling is $M$. This gives a holomorphic filling $W$ of $M$ with $M=\partial W$. However, to effectively compute invariants we need a stronger condition on $W$, for example, when $W$ is a Stein manifold. This happens for hypersurfaces of weighted singularities, in particular for Brieskorn manifolds, since we can smooth the singularity at $0\in Y$ to give a Stein manifold $W^\infty$. In this case the Liouville vector field $\Psi$ is globally defined on $W^\infty$. We then obtain a Stein filling $W$ by cutting off $W^\infty$ at $r=1$ and identifying $M$ with $W^\infty\cap\{r=1\} =\partial W$. 

The invariants are constructed from Floer theory on the free loop space $\grL W^\infty$ parameterized by the sphere $S^{2N+1}$ by considering the action functional 
$$\cala_N:W^\infty\times S^{2N+1}\ra{1.6} \bbr$$ 
defined by
\begin{equation}\label{Floeract}
\cala_N(\grg,z)=-\int_{S^1}\grg^*\eta -\int_0^1H(t,\grg(t),z)dt
\end{equation}
where $\grg:S^1\ra{1.6} W^\infty$ is a loop and $H(t,\grg(t),z)$ is a time dependent $S^1$ invariant Hamiltonian on $W^\infty\times S^{2N+1}$. The critical points of $\cala_N$ are periodic solutions to the parameterized Hamilton's equations 
\begin{equation}\label{parHam}
\dot{x}=X_H(x),\qquad \int_0^1 \vec \nabla_z H(t,x(t),z)dt=0
\end{equation}
where  $X_H$ is the Hamiltonian vector field associated to the Hamilton $H$. From the negative gradient flow of $\cala_N$ one obtains the parameterized Floer equations
for $\bar{u}=(u,z):\bbr\times S^1\ra{1.8} W^\infty\times S^{2N+1}$ given by
\begin{equation}
\label{eq:param_floer}
\begin{split}
\frac{\partial}{\partial s} u+I(t,u,z)(\frac{\partial u}{\partial t}-X_H) &=0\\
\frac{d}{ds}z-\int_{0}^1 \vec \nabla_z H(t,u(s,t),z(s))\,dt &=0\\
\lim_{s\to \mp \infty} \bar u(s,t)&\in S_{\pm}
\end{split}
\end{equation}
where $I(t,u,z)$ is a complex structure on the Stein manifold $W^\infty$, and $S_{\pm}$ are $S^1$-orbits of critical points of $\mathcal A_N$. Suffice it to say (see Section 4.2 of \cite{BMvK15} for details) that from this data one obtains an $S^1$ equivariant Floer chain complex $SC^{S^1,N}$ and from this its $S^1$ equivariant homology $SH^{S^1,N}(W)$. However, this homology depends on the choice of Hamiltonian as well as on $N$. The dependence on $N$ can be removed by taking the direct limit as $N\ra{1.6}\infty$. To remove the dependence on the choice of Hamiltonian we use Lemma 5.6 of \cite{BoOa13a} to give a smooth homotopy of Hamiltonians. This gives the $S^1$ equivariant symplectic homology $SH^{S^1}(W)$. Then, truncating the action $\cala$ to its positive part gives $SH^{S^1,+}(W)$. We have a theorem of Gutt \cite{Gutt14,Gutt17}

\begin{theorem}[Gutt]\label{KvKthem}
Let $M$ be a deformation of a Brieskorn manifold. Then $SH^{S^1,+}(W)$ is invariant under deformations of the symplectic structure.
\end{theorem}
There is a more general result in \cite{Gutt17}. By a deformation of a Brieskorn manifold we mean by adding monomials of $\deg f$ to the Brieskorn-Pham polynomial representing $M$. The Betti numbers of the positive part of $S^1$ equivariant symplectic homology $SH^{S^1,+}(W)$ are $sb_i={\rm rank}~SH^{S^1,+}_i(W)$. We define the {\it mean Euler characteristic} of $W$ as
\begin{equation}
\label{eq:def_mec}
\chi_m(W) = 
\frac{1}{2} 
\left( 
\liminf_{N \to \infty}  \frac{1}{N} \sum_{i=-N}^N (-1)^i sb_i(W) \,+\,
\limsup_{N \to \infty}  \frac{1}{N} \sum_{i=-N}^N (-1)^i sb_i(W)
\right) 
\end{equation}
if it exists. The following corollary, which follows from Proposition 4.21 of \cite{BMvK15}, was suggested by Otto van Koert 

\begin{corollary}
Suppose that $M$ is the link of a weighted homogeneous polynomial $f$ with an isolated singularity at $0$, and denote the smoothing $f^{-1}(1)$ by $W$. Assume that $\dim(M)>1$.
If the quotient orbifold $Q=M/S^1$ is Fano or of general type, then $\chi_m(W)$ is a contact invariant.
\end{corollary}

For details regarding the invariance and computation of the mean Euler characteristic we refer to Lemma 5.15 and Appendix C in \cite{KwvKo13}.


In Example 5.1 and Lemma 5.2 of \cite{BMvK15} an example of contact manifolds that cannot be distinguished by $\chi_m(W)$ but can by their homology groups $SH^{S^1,+}$ is given.

A special case of much interest is when $c_1(\cald)=0$. We note that any smooth link of a complete intersection by weighted homogeneous polynomials of dimension $2n+1$ is $n-1$-connected and has $c_1(\cald)=0$.

\section{Extremal Sasaki Geometry}\label{extsassec}
The notion of extremal K\"ahler metrics was introduced as a variational problem by Calabi in \cite{Cal56} and studied in greater depth in \cite{Cal82}. The most effective functional is probably the $L^2$-norm of scalar curvature, viz.
\begin{equation}\label{Calfun}
E(\omega) = \int_M s^2 d\mu,
\end{equation}
where $s$ is the scalar curvature and $d\mu$ is the volume form of the
K\"ahler metric corresponding to the K\"ahler form $\omega$. The variation is taken over the set of all K\"ahler metrics within a fixed K\"ahler class $[\gro]$. Calabi showed that the critical points of the functional $E$ are precisely the K\"ahler metrics such that the gradient vector field $J{\rm grad}~s$ is holomorphic. Recent detailed accounts of K\"ahler geometry are given in \cite{Gau09b,Sze14}.

On the Sasaki level extremal metrics were developed in \cite{BGS06}. The procedure is quite analogous again using the $L^2$-norm of scalar curvature $s_g$ of the Sasaki metric $g$, viz.
\begin{equation}\label{sasfun}
E(g) = \int_M s_g^2 dv_g,
\end{equation}
where now the variation is taken over the space ${\mathcal S}(M,R,\bar{J})$. Actually by the Sasaki version of a theorem of Calabi \cite{Cal85} extremal Sasaki metrics have maximal symmetry, so we can take the variation over the subspace of $\bbt$-invariant functions ${\mathcal S}(M,R,\bar{J})^\bbt$.


Again the critical points are precisely those Sasaki metrics such that the $(1,0)$ component of the gradient vector field $\partial^{\#}s_g$ is transversely holomorphic. Since the scalar curvature $s_g$ is related to the transverse scalar curvature $s^T_g$ of the transverse K\"ahler metric by $s_g=s_g^T-2n$, a Sasaki metric is extremal (CSC) if and only if its transverse K\"ahler metric is extremal (CSC). It is often more convenient to deal with the transverse scalar curvature $s^T$. 

\begin{remark}\label{negnul}
In the case of negative or null type the connected component of the automorphism group $\gA\gu\gt(\cals)$ is the circle group $S^1$ generated by the Reeb vector field. So any extremal Sasaki metric of negative or null type has constant scalar curvature. The existence of constant scalar curvature metrics on such negative Sasaki manifolds is generally still an open question. However, it is known when $c_1(\cald)=0$ by the transverse Aubin-Yau theorem \cite{ElK} that constant scalar curvature Sasaki metrics do exist. They are called Sasaki-eta-Einstein. Moreover, for null Sasakian structures we have

\begin{proposition}\label{nulcsc}
Every class in $\gM^{cl}_{N}$ has a constant scalar curvature representative $g$ which is eta-Einstein satisfying
$$\Ric_g=-2g+2(n+1)\eta\otimes \eta.$$
\end{proposition}

\end{remark}

Thus, we shall focus on Sasakian structures of positive or indefinite type.
We recall that $\mathfrak{h}^T(R,\bar{J})$ denotes  the Lie algebra of transversely holomorphic vector fields which is infinite dimensional owing to arbitrary sections of the line bundle $L_R$ generated by $R$. We have the following equivalent characterizations

\begin{theorem}\label{extremthm}
Let $(M,\cals)$ be a Sasaki manifold. Then the following are equivalent
\begin{enumerate}
\item $\cals$ is extremal;
\item The gradient vector field $\partial^{\#}s_g$ of the scalar curvature $s_g$ is transversely holomorphic;
\item $s_g$ is a solution to Equation \eqref{ord4};
\item $s_g\in \calh^\cals_B$;
\item there is a maximal $k$-dimensional torus $\bbt$ in $\gC\go\gn(M,\eta)$ such that $s_g\in \calh^\bbt_B$.
\end{enumerate}
\end{theorem}

In K\"ahler geometry there is an important invariant that obstructs CSC K\"ahler metrics, the Futaki invariant which is a character on the Lie algebra of holomorphic vector fields. The search for canonical K\"ahler metrics either extremal or when possible CSC K\"ahler metrics has led to the geometric invariant theory notion of K-stability introduced by Tian \cite{Tia97} and developed further by Donaldson \cite{Don02} and his school, cf. \cite{Sze14} and references therein. Further, a functional whose critical points are CSC K\"ahler metrics was introduced and studied by Mabuchi \cite{Mab86}.
Extremal K\"ahler metrics are related to the notion of relative stability described by Sz\'ekelyhidi \cite{Sze06,Sze07}.
There are analogous notions in Sasaki geometry which we now explore.

\subsection{The Sasaki-Futaki Invariant and K-Stability}
Recall from \cite{BGS06} that the Sasaki--Futaki invariant of the Reeb vector field $R$ on $M$, is the map 
$$ {\bf F}_R : \mathfrak{h}(R,\bar{J}) \longrightarrow \bbc$$ 
defined by
\begin{equation}\label{SFinv}
{\bf F}_R(X)= \int_MX\psi_g dv_g,
\end{equation} 
where $\psi_g$ is the unique basic function of average value $0$ that satisfies 
$$\rho^T = \rho^T_H + i\partial\overline{\partial} \psi_g$$ 
where $\rho^T_H$ is $\Delta_B$--harmonic. However, since ${\bf F}_R(R)=0$ we can consider ${\bf F}_R$ as a character on the finite dimensional Lie algebra $\bar{\mathfrak{h}}(R,\bar{J}) =\mathfrak{h}(R,\bar{J})/\grG(L_R)$.
We also mention that ${\bf F}_R$ is independent of the representative in $\Sas(R,\bar{J})$ and that ${\bf F}_R([X,Y])=0$, see~\cite{BGS06,FOW06}. As with the Futaki invariant in K\"ahler geometry \cite{Fut83}, ${\bf F}_R$ is an obstruction to the existence of constant scalar curvature Sasaki metrics (cscS). 

The stability properties for quasi-regular Sasaki metrics are equivalent to the stability properties of certain K\"ahler orbifolds which was investigated in detail in \cite{RoTh11}. However, to treat the general Sasaki metrics one needs to work on the affine cone as described in Section \ref{affcone}. This was done by Collins and Sz\'ekelyhidi \cite{CoSz12} and it is this approach that we follow here. We begin by defining two important functionals, the {\it total volume} and the {\it total transverse scalar curvature}, viz. 
\begin{equation}\label{VS}
\bfV_R=\int_Mdv_g,\qquad \bfS_R=\int_Ms^Tdv_g
\end{equation}
where $s^T$ denotes the transverse scalar curvature of the Sasakian structure $\cals=(R,\eta,\Phi,g)$. We define the {\it Donaldson-Futaki invariant} on the polarized affine cone $(Y,R)$ by
\begin{equation}\label{DonFut}
{\rm Fut}(Y,R,a) := \frac{\bfV_R}{n}  D_{a}\left( \frac{\bfS_R}{\bfV_R} \right)+ \frac{\bfS_R D_{a}\bfV_R}{n(n+1)\bfV_R}.
\end{equation}
A straigthforward computation \cite[Lemma 2.15]{TivC15} shows that there exists a constant $c_n>0$ depending only on the dimension $n$ such that 
\begin{equation}\label{Futeqn}
Fut(Y,R,a)=c_n \bfF_{R}(\Phi(a)).
\end{equation} 
To proceed further we need the definition of a special type of degeneration known as a {\it test configuration} due to Donaldson in the K\"ahler case and Collins-Sz\'ekelyhidi in the Sasaki case:

\begin{definition}\label{testconf}
Let $(Y,R)$ be a polarized affine variety with an action of a torus $T^{\bbc^*}$ for which the Lie algebra $\gt$ of a maximal compact subtorus $T^\bbr$ contains the Reeb vector field $R$. A $T^{\bbc}$-equivariant test configuration for $(Y,R)$ is given by a set of $k$ $T^{\bbc}$-homogeneous generators $f_1,\dots, f_k$ of the coordinate ring $\calh$ of $Y$ and $k$ integers $w_1, \ldots, w_k$ (weights). The functions $f_1,\dots, f_k$ are used to embed $Y$ in $\bbc^k$ on which the weights $w_1, \ldots, w_k$ determine a $\bbc^*$ action. By taking the flat limit of the orbits of $Y$ to $0\in \bbc$ we get a family of affine schemes $\mathcal{Y} \longrightarrow \bbc.$ There is then an action of $\bbc^*$ on the `central fiber' $Y_0$, generated by $a \in\gt'$, the Lie algebra of  some torus $T'^\bbc \subset {\rm GL}(k,\bbc)$ containing $T^\bbc$. 
\end{definition}

See \cite{Sze14} for more details including the precise definition of flat limit. Note that the Reeb  vector field $R$ for $Y$ is also a Reeb vector field for the central fiber $Y_0$. We now can obtain the correct notion of stability.

\begin{definition}\label{stabdef}
We say that the polarized affine variety $(Y,R)$ is $K$-semistable if for each $T^\bbc$ such that $R \in\gt$ the Lie algebra of $\bbt^\bbr$ and any $T^\bbc$-equivariant test configuration we have
$${\rm Fut}(Y,R,a)\geq 0$$
where $a\in \gt'$ is the infinitesimal generator of the induced $S^1$ action on the central fiber $Y_0$. The polarized variety $(Y,R)$ is said to be K-polystable if equality holds only for the product configuration $\mathcal{Y}=Y\times\bbc$.
\end{definition}

We now have a result of Collins and Sz\'ekelyhidi
\begin{theorem}[\cite{CoSz12}]\label{CoSzthem} 
Let $(M,\cals)$ be a Sasaki manifold of constant scalar curvature. Then its polarized affine cone $(Y,R)$ is K-semistable.
\end{theorem}

The proof of this theorem makes use of a certain Hilbert series, namely the {\it index character} 
\begin{equation}\label{indexchar}
F(R,t)=\sum_{\gra\in\gt^*}e^{-t\gra(R)}\dim\calh_\gra 
\end{equation}
introduced in \cite{MaSpYau06} and developed further in \cite{CoSz12}. In the latter it was shown that $F(R,t)$ has the meromorphic extension
$$F(R,t)=\frac{n!a_0(R)}{t^{n+1}}+\frac{(n-1)!a_1(R)}{t^n}+\cdots $$
where $a_0(R)=\bfV_R$ and $a_1(R)=\bfS_R$ up to constants which are irrelevant to the argument. The author and his colleagues \cite{BHLT15} studied the stability problem through the use of the Einstein-Hilbert functional which I discuss in the next section.

The Yau-Tian-Donaldson conjecture in the Sasaki case is:

\begin{conjecture}\label{YTDconj}
A Sasaki manifold $(M,\cals)$ has constant scalar curvature if and only if its polarized affine cone $(Y,R)$ is K-semistable.
\end{conjecture}

Theorem \ref{CoSzthem} proves one direction. The other direction is still open in general, but is known to hold in certain special cases. It has been proven in the $\bbq$-Gorenstein case $c_1(\cald)^r=0$ for some $r<n+1$ by Collins and Sz\'ekelyhidi \cite{CoSz15}. A case when $c_1(\cald)\neq 0$ is given below in Corollary \ref{BHLTcor}.

\subsection{The Einstein-Hilbert Functional}
We define the Einstein-Hilbert functional 
\begin{equation} 
\bfH(R) = \frac{\bfS^{n+1}_{R}}{\bfV_R^n}
\end{equation} 
as a functional on the Sasaki cone. Note that $\bfH$ is homogeneous since the rescaling $R\mapsto a^{-1}R$ gives $dv_g\mapsto a^{n+1}dv_g$ and $s_g^T\mapsto a^{-1}s^T_g.$ So $\bfH$ is only a function of the ray in $\gt$. Moreover, it follows from the invariance of $\bfV_R$ \cite{BG05} and $\bfS_R$ \cite{FOW06} that

\begin{lemma}\label{invisot}
The Einstein-Hilbert functional $\bfH$ only depends on the isotopy class in $\gB\gM^{cl}_\cals$.
\end{lemma}

The lemma allows us to consider extremality as a property of the isotopy class, that is $\bar{\cals}$ has an extremal representative. Similar considerations hold for constant scalar curvature Sasaki metrics. We are interested in the critical points of the functional $\bfH$. A main result of \cite{BHLT15} is the variational formula

\begin{lemma}\label{lemmaCRIT} Given $a\in T_R\gt_k^+$, we have $$d\bfH_R (a) = \frac{n(n+1)\bfS^{n}_{R}}{\bfV_R^n} \bfF_{R}(\Phi(a)).$$ If $\bfS_{R} =0$ then $d\bfS_{R}= n \bfF_{R}(\Phi(a)).$ 
\end{lemma}

Actually, it is more convenient to work with the `signed' version
\begin{equation}\label{signEH}
{\bf H}_1(R)= {\rm sign}({\bf S}_R)\frac{|{\bf S}_R|^{n+1}}{{\bf V}^n_R}.
\end{equation}

In either case the set of critical points is the union of the zeroes of $\bfS_R$ and the zeroes of Sasaki-Futaki invariant $\bfF_{R}$. Thus, if the Sasakian structure given by $R$ is extremal any critical point with $\bfS_R\neq 0$ must have constant scalar curvature \cite{Fut83,BGS06}. Using Lemma \ref{lemmaCRIT}, Equation \eqref{Futeqn} and the linearity of Equation \eqref{DonFut} in $a$ gives

\begin{theorem}[\cite{BHLT15}]\label{BHLTthm}
Let $(M,\cals)$ be a Sasaki manifold such that its polarized affine cone $(Y,R)$ is K-semistable. Then $\bfF_R$ vanishes and $R$ is a critical point of $\bfH_1$. So if $\cals$ is extremal, then it has constant scalar curvature.
\end{theorem}

From this and Theorem 1.2 of \cite{BoTo14a} we have

\begin{corollary}\label{BHLTcor}
Let $M_{\bfl,\bfw}=M\star_\bfl S^3_\bfw$ be an $S^3_\bfw$ join with a regular Sasaki manifold with constant transverse scalar curvature $s^T\geq 0$ and no transverse Hamiltonian vector fields. Then a Sasakian structure $\cals$ in the family $\{M_{\bfl,\bfw}\}$ has constant scalar curvature if and only if its polarized affine cone $(Y,R)$ is K-semistable.
\end{corollary}

Theorem \ref{BHLTthm} implies that if $R$ is not a critical point of $\bfH_1$ then it is not K-semistable, hence it is K-unstable. We now discuss further properties of $\bfH_1$. This functional was first studied in the $\bbq$-Gorenstein case \cite{MaSpYau06} ($c_1(\cald)^r=0$) where it was shown to be proportional to the volume functional $\bfV_R$, and that the volume $\bfV_R$ is globally convex. This is certainly not true generally when $c_1(\cald)\neq 0$ \cite{Leg10,BoTo14a} as we give explicit examples later.

Nevertheless, applying the well known Duistermaat-Heckman theorem to Sasaki geometry, the authors \cite{BHL17} proved 

\begin{theorem}[\cite{BHL17}]\label{BHLthm}
Let $(N,\cald)$ be a contact manifold of Sasaki type with the action of a torus $\bbt$ of Reeb type. Then 
\begin{enumerate}
\item $\bfS_R,\bfV_R,\bfH(R)$ are rational functions\footnote{This result can also be obtained from the Hilbert series of the index character as explained to us by Tristan Collins in a private communication.} of $R$;
\item $\bfV_R,\bfS_r,\bfH,\bfH_1(R)$ all tend to $+\infty$ as $R$ approaches the boundary of the Sasaki cone $\gt^+$ (away from $0$);
\item there exists a ray $\gr_{min}\subset \gt^+$ that minimizes $\bfH_1$;
\item if $\dim \bbt> 1$ and $H_1(\gr_{min})\leq 0$, all Sasakian structures in $\gt^+$ are indefinite;
\item if $\bfH_1(R_{min})\neq 0$ and the corresponding Sasaki metric is extremal, it must have constant scalar curvature.
\end{enumerate}
\end{theorem}

\subsection{Relative K-stability}
Relative stability in the Sasaki case \cite{BovC16} follows closely the K\"ahler case \cite{Sze06,Sze07,StSz11} with the caveat that one makes use of the index character $F(R,t)$ \eqref{indexchar} in the relevant proofs. For the relative version of Definition \ref{stabdef} we need the transverse Futaki-Mabuchi vector field $\chi$ \cite{FuMa95} with its bilinear form $\langle\cdot,\cdot \rangle_\chi$. We then define the \emph{Donaldson-Futaki invariant relative to $T$} of a test configuration
\[  Fut_{\chi}(Y_0, \xi,\zeta)= Fut(Y_0, \xi,\zeta) -\langle \zeta,\chi\rangle.\] 
We now have
\begin{definition}
A polarized affine variety  $(Y,\xi)$ with a unique singular point is \emph{K-semistable relative to $T$} if for every T-equivariant test configuration
\[ {\rm Fut}_\chi(Y_0,R,a)\geq 0.\]
\end{definition}

From the Sasaki version \cite{CoSz12} of the Donaldson Lower Bound Theorem one obtains

\begin{theorem}[\cite{BovC16}]\label{Bovthm}
Let $\cals$ be an extremal Sasaki structure on $M$. Then the polarized affine variety $(Y,R)$ is K-semistable relative to a maximal torus $\bbt$. 
\end{theorem}

So relative K-semistability provides an obstruction to the extremality of a Sasakian structure. Examples in the Gorenstein case can be easily obtained from Brieskorn manifolds and their perturbations. Consider the Brieskorn link \eqref{Brieslink} associated to the polynomial
\begin{equation}\label{Feqn}
f(z_0,\ldots,z_n)=f_1(z_0,\ldots,z_k)+z_{k+1}^2+\cdots +z_n^2
\end{equation}
with $n-k\geq 2$ and all weights $w_i$ with $i=0,\ldots,k$ satisfy $2w_i<d_1$, the degree of $f_1$.
In this case the connected component of the Sasaki automorphism group is
$U(1)\times SO(n-k)$, so the Sasaki cone $\gt^+$ has dimension $r+1$ with $r=\lfloor\frac{n-k}{2}\rfloor$, and is given by
\begin{equation}\label{Fsascone}
\gt^+=\{b_0R_\bfw+\sum_{j=1}^rb_j\grz_j\in \gt~|~b_0>0,~-\frac{db_0}{4}<b_j<\frac{db_0}{4}\}
\end{equation}
where $R_\bfw$ is the standard Reeb field on $L_f$. Defining variables
$u_j=z_{k+j}+iz_{k+j+1}$ and $v_j=z_{k+j}-iz_{k+j+1}$ we can write the quadratic part of \eqref{Feqn} as $\sum_{j=1}^r u_jv_j~ \text{if $n-k$ is even}$, and $\sum_{j=1}^ru_jv_j +z_n^2  ~\text{if $n-k$ is odd.}$ The point being that vector field $\grz_j$ has weight $(1,-1)$ with respect to $(u_j,v_j)$ and $0$ elsewhere. It follows \cite{BovC16} that the well known Lichnerowicz obstruction of Gauntlett, Martelli, Sparks, and Yau \cite{GMSY06} obstruction of SE metrics for $R_\bfw$ gives obstructions to the existence of extremal Sasaki metrics for every $R\in \gt^+$. This gives a large class of Sasakian structures on the links of some Brieskorn-Pham hypersurfaces, shown in Table \ref{sasconetable}, where extremality is obstructed.

\begin{table}[h]
$\begin{array}{|c|c|c|}
\hline
\text{Diffeo-(homeo)-morphism Type}        & f &  \dim \gt^+ \\\hline
S^{2n}\times S^{2n+1}  & z_0^{8l}+z_1^2+\cdots +z_{2n+1}^2,\ n,l\geq1   &  n+1 \\
S^{2n}\times S^{2n+1}\#\grS^{4n+1}_1 & z_0^{8l+4}+z_1^2+\cdots +z_{2n+1}^2=0,\ n\geq1,l\geq 0  & n+1 \\
\text{Unit tangent bundle of $S^{2n+1}$}  & z_0^{4l+2}+z_1^2+\cdots +z_{2n+1}^2, \ n>1,l\geq 1 &  n+1 \\
\text{Homotopy sphere}~\grS_k^{4n+1}  & z_0^{2k+1}+  z_1^2 +\cdots +z_{2n+1}^2, \ n>1, k\geq 1 &  n+1 \\
\text{Homotopy sphere}~\grS_k^{4n-1}  & z_0^{6k-1}+  z_1^3 +z_2^2 +\cdots +z_{2n}^2, \ n\geq 2,k\geq 1 &  n \\
\text{Rat. homology sphere}~H_{2n}\approx\bbz_3  & z_0^{k}+  z_1^3 +\cdots +z_{2n}^2,\ n,k>1 &  n \\
2k(S^{2n+1}\times S^{2n+2}),~~D_{n+1}(k)    & z_0^{2(2k+1)} +z_1^{2k+1}+ z_2^2 +\cdots +z_{2n+2}^2, \ n,k\geq 1 & n+1 \\
\#m(S^2\times S^3), m=\gcd(p,q)-1    &  z_0^p+z_1^q+z_2^2+z_3^2,~ \ p\geq 2q~\text{or}~q\geq 2p  &   2 \\
\hline
\end{array}$
\vspace{.1in}
\caption{Manifolds having Sasaki Cones with no Extremal Metrics}
\label{sasconetable}
\end{table}

\section{Examples of Moduli Spaces for Dimension 5}\label{Exmodsect}
In this section we consider examples of our various associated moduli spaces for 5-dimensional Sasaki manifolds. A first question may be: what is $\pi_0(\gM^{cl}_\grg)$?
For example we have the following from \cite{BMvK15} giving examples of 5-manifolds whose moduli space of positive Sasaki classes have an infinite number of components:

\begin{theorem}\label{ks2s3comp}
For each $k=1,2,\dots$ we have $|\pi_0(\gM_{+,0}^{cl}(k(S^2\times S^3)))|=\aleph_0$. Moreover, for the rational homology 5-spheres 
$$M=M_2,M_3,M_5,2M_3,4M_2$$ 
we have $|\pi_0(\gM_{+,0}^{cl}(M))|=\aleph_0$.
\end{theorem}

We shall consider the moduli spaces for $k=0,1$ in much more detail below.
Theorem \ref{ks2s3comp} can be contrasted with the following simply connected rational homology spheres  whose moduli have a single component:

\begin{theorem}[Koll\'ar \cite{Kol05b,Kol06}]\label{Kolthm}
Let $lM_k$ denote the $l$-fold connected sum of the simply connected rational homology sphere $M_k$ with $H^2(M_k,\bbz)=\bbz_k\oplus \bbz_k$. Then
\begin{enumerate}
\item $\gM_{+}^{cl}(2M_5)=\gM^{SE}(2M_5)\times \bbr$ where $\gM^{SE}(2M_5)$ has $\dim_\bbr=6$ parameterized by the moduli space of genus 2 curves.
\item $\gM_{+}^{cl}(4M_3)=\gM^{SE}(4M_3)\times \bbr$ where $\gM^{SE}(4M_3)$ has $\dim_\bbr=14$ parameterized by the moduli space of hyperelliptic genus 4 curves.
\end{enumerate}
\end{theorem}

Although generally the complexity is not a contact invariant \cite{Boy10a}, it is convenient in dimension five for matters of discussion to decompose $\gM^{cl}_\grg$ into pieces according to complexity. 


\subsection{Moduli of Sasakian structures on $S^5$}
Since every contact structure on $S^5$ satisfies $c_1(\cald)=0$ we are dealing with $\gM_0^{cl}$ so we can drop the subscript $0$. Furthermore, since there are no null Sasakian structures on $S^5$ (\cite{BG05}, Corollary 10.3.9), we have 
$$\gM^{cl}(S^5)=\gM_{+}^{cl}(S^5)\cup \gM_{-}^{cl}(S^5).$$ 
Although there are infinitely many distinct Sasakian classes in $\gM_{-}^{cl}(S^5)$, we only treat the space $\gM_{+}^{cl}(S^5)$ here. We do not yet have a complete picture of $\gM_{+}(S^5)$, but $\chi_m(W)$ and $SH^{+,S^1}(W)$ allow us to determine infinitely many components. Moreover, we know from Example \ref{standardsphere} that  $\gM_{+}(S^5)$ contains the subspace of standard Sasakian structures $\gM^{cl}_o(S^5)=\grk_3(\cald_o,J_o)$. We view this as the local moduli space obtained from deformations.

\begin{theorem}\label{s5thm}
The moduli space $\gM_{+}^{cl}(S^5)$ has a unique toric component consisting of $\gM^{cl}_o(S^5)$. Moreover, all complexity 1 components are exotic and are represented by the link of the Brieskorn-Pham polynomial $f_{p,q}=z_0^p+z_1^q+z_2^2+z_3^2$ with $p,q>1$ and $\gcd(p,q)=1$, and there are infinitely many such components. Furthermore, there are infinitely many exotic complexity 2 components that are not contactomorphic to any complexity 1 component.
\end{theorem}

\begin{proof}
The only toric contact structure of Reeb type on $S^5$ is the standard contact structure. A result of Liendo and S\"uss \cite{LiSu13,Suss18} says that every complexity 1 component is represented by the link of the Brieskorn-Pham polynomial $f_{p,q}$ given in the hypothesis, and by Proposition \ref{Briesexotic} these are all exotic. Furthermore, by \cite{BMvK15} the mean Euler characteristic of these links is given by
$$\chi_m=\frac{pq+1}{2(p+q)}.$$
So there are infinitely many non-contactomorphic components of complexity 1. For the complexity 2 components it is known from  \cite{Ust99,KwvKo13,Gutt17}that there is a countable infinity of complexity two components of $\gM_{+,0}^{cl}(S^5)$. Explicitly consider the Brieskorn-Pham link of $z_0^{1+30k}+z_1^5+z_2^3+z_3^2$ which has mean Euler characteristic  \cite{BMvK15}
$$\chi_m=\frac{31+270k}{62+60k}.$$
Moreover, we claim that there are infinitely many complexity 2 components that are non-contactomorphic to any complexity 1 component. To see this we note that equality of the two mean Euler characteristics implies
$$31(pq+1-p-q)=-30k(pq+1-9p-9q).$$
This equation has no positive integer solutions for $k$ when $p,q\geq 18$ and for $1\leq p,q<18$ there are at most a finite number of solutions. This concludes the proof. 
\end{proof}

\begin{remark}
We do not know all complexity 2 components on $S^5$.  There are many that can be represented by hypersurface singularities as well as complete intersection singularities \cite{Inb13}. 
\end{remark}

Turning to the moduli space $\gM^{SE}(S^5)$ of Sasaki-Einstein metrics on $S^5$ we have combining results in \cite{BGK05,LiSu13,BMvK15,BovC16,CoSz15,Suss18} 

\begin{theorem}\label{SEs5}
The moduli space $\gM^{SE}(S^5)$ satisfies
\begin{enumerate}
\item the moduli space  $\gM^{SE}(S^5)_o$ of toric SE metrics in $\gM^{cl}_o(S^5)$ consists of one point;
\item There are an infinite number of inequivalent components of complexity 1 SE moduli spaces consisting of a single point. All components are given by the link of the Brieskorn-Pham polynomial $f_{p,q}$ of Theorem \ref{s5thm} which satisfy the inequalities $2p>q$ and $2q>p$. Furthermore, if these inequalities are not satisfied then there are no such SE metrics nor are there any extremal Sasaki metrics in the corresponding Sasaki cone, and all of these components are distinct from $\gM^{SE}(S^5)_o$;
\item For complexity 2 there are at least 55 components of $\gM^{SE}$ with $\dim_\bbr=0,$ 18 components with $\dim_\bbr\gM^{SE}=2$, 4 components with $\dim_\bbr\gM^{SE}=4$, one component with $\dim_\bbr\gM^{SE}=6$, 2 components with $\dim_\bbr\gM^{SE}=8$, one component with $\dim_\bbr\gM^{SE}=10$, and one with $\dim_\bbr\gM^{SE}=20$. 
\end{enumerate}
\end{theorem}

\begin{remark}
A recent result of S\"uss says that there are no irregular SE metrics on $S^5$ \cite{Suss18}.
\end{remark}




\subsection{Moduli of toric Sasakian Structures on $S^3$ Bundles over $S^2$}\label{s2s3sect}
It is well known that there are precisely two (linear) $S^3$ bundles over $S^2$, namely the trivial bundle $S^2\times S^3$ with vanishing second Stiefel-Whitney class $w_2$, and the non-trivial bundle $S^2\tilde{\times}S^3$ with non-vanishing $w_2$.
In this case we can have $c_1(\cald)\neq 0$, but again there are no null Sasakian structures on $S^2\times S^3$ or $S^2\tilde{\times}S^3$. There are, however, negative Sasakian structures, in particular on $S^2\times S^3$ with $c_1(\cald)=0$, but it still appears to be open whether or not there are infinitely many. Here we concentrate on the positive cases 
$$\gM^{cl}_{\grg,+}(S^2\times S^3)~~ \text{and}~~\gM^{cl}_{\grg,+}(S^2\tilde{\times} S^3) .$$
It follows from contact reduction theory \cite{BG05} that all toric contact structures $\cald$ on $S^2\times S^3$ or the non-trivial bundle $S^2\tilde{\times}S^3$ are obtained by a circle action on the standard $S^7$. We refer to \cite{BoPa10} for details.
If we let $\grg$ denote the positive generator in $H^2(S^2\times S^3)$ and $H^2(S^2\tilde{\times} S^3)$, we see that $c_1(\cald)=(p_1+p_2-p_3-p_4)\grg$ where $p_i\in\bbz$. We denote these toric contact structures as $\cald_\bfp$. In sharp contrast with $S^5$ there are infinitely many toric contact structures on $S^2\times S^3$. We can also contrast this with the symplectic case \cite{KaKePi07}.
 

The following theorem is a composite of results in \cite{BGS06,BoPa10,BoTo14a,BHL17}.

\begin{theorem}\label{torics2s3}
Let $(M,\cald_\bfp)$ be a toric contact structure on the $S^3$ bundle over $S^2$ with total space $M$. Then
\begin{enumerate}
\item there is a 4-parameter family of toric contact structures on $S^2\times S^3$ and $S^2\tilde{\times}S^3$ depending on whether $p_1+p_2-p_3-p_4$ is even or odd;
\item for each such toric contact structure $\cald_\bfp$ the set of positive Sasakian structures $\gp^+$ is a non-empty open subset in $\gt^+$, and for every $R\in\gt^+$ we have $\bfS_R>0$; 
\item for each such toric contact structure $\cald_\bfp$ the $\bfw$ subcone $\gt^+_\bfw$ is saturated by extremal Sasaki metrics; hence, the extremal subset $\ge$ is an open subset of $\gt^+$ containing $\gt^+_\bfw$. 
\end{enumerate}
\end{theorem}

The question as to which toric contact structures are inequivalent as contact structures is in general a very difficult one. Of course, distinct values of $p_1+p_2-p_3-p_4$, up to sign, give inequivalent contact structures. In \cite{BoPa10} inequivalence of a subclass of these contact structures with the same $p_1+p_2-p_3-p_4$ was shown under the assumption that contact homology is well defined. Unfortunately, in general this is still open; however, there is a recent proof by Abreu and Macarini \cite{AbMa18} when $p_1+p_2-p_3-p_4=0$ which we briefly discuss later.

When $c_1(\cald_\bfp)\neq 0$ the preferred metrics are those of constant scalar curvature (CSC). These come in rays in $\gt^+$. The toric case was extensively studied by Legendre \cite{Leg09,Leg10} where it is proven that there are at most a finite number of CSC rays in the Sasaki cone $\gt^+$ of any $S^3$ bundle over $S^2$. Furthermore, she gives explicit examples of more than one CSC ray. Summarizing 

\begin{theorem}[Legendre]\label{Legthm}
A co-oriented toric contact manifold admits at least one ray with vanishing Sasaki-Futaki invariant. Moreover, for a co-oriented toric contact structure on an $S^3$ bundle over $S^2$ admits at least one CSC ray and at most 7 CSC rays.
\end{theorem}

\begin{remarks}
\begin{enumerate}
\item Recall that generally the vanishing of the Sasaki-Futaki invariant does not imply CSC. It does imply CSC if a priori the ray is extremal.
\item There are many examples of 2 inequivalent CSC rays.
\item It should be possible to lower the bound from 7.
\end{enumerate}
\end{remarks}

We now turn to the monotone case when $c_1(\cald_\bfp)=0$. In this case we have the trivial bundle $S^2\times S^3$. Clearly, the 4-parameter family reduces to a 3-parameter family of toric contact structures on $S^2\times S^3$. Furthermore, it was shown in \cite{CLPP05,MaSp05b} that for every such toric $(\cald_\bfp,J)$ satisfying $p_1+p_2-p_3-p_4=0$ there is a unique SE metric. We have a much better understanding of the moduli spaces for a subclass of the toric contact structures, denoted by $Y^{p,q}$ where $\gcd(p,q)=1$ and $1\leq q<p$, of these SE metrics which were discovered earlier \cite{GMSW04a} and are more amenable to further study. In \cite{BoTo14a} it was shown that the $Y^{p,q}$ can be obtained as an $S^3_\bfw$ join with the standard $S^3$, and that the SE metric lies in the $\gt^+_\bfw$ subcone of $\gt^+$. Writing $Y^{p,q}=S^3\star_{l_1,l_2}S^3_\bfw$ we see that 
$$l_1=\gcd(p+q,p-q),~l_2=p,~ \bfw=\frac{1}{\gcd(p+q,p-q)}(p+q,p-q).$$
In \cite{BoPa10} it was shown that although $Y^{p,q}$ and $Y^{p,q'}$ are inequivalent as toric contact structures, they are equivalent as $T^2$ equivariant contact structures, see also the survey \cite{Boy11} for more details specific to $Y^{p,q}$. The fact that $Y^{p,q}$ and $Y^{p,q'}$ are $T^2$ equivariantly equivalent, but not $T^3$ equivariantly equivalent is related to the fact that their maximal tori belong to different conjugacy classes in the contactomorphism group $\gC\go\gn(M,\cald)$.

In \cite{AbMa10} as well as in \cite{BoPa10} it was shown that $Y^{p,q}$ and $Y^{p',q'}$ are contact inequivalent if $p'\neq p$ whenever the contact homology is well defined. Unfortunately, the well definedness of contact homology is still generally unknown; however, fortunately, as mentioned previously, Abreu and Macarini \cite{AbMa18} have recently given a proof implying that $Y^{p,q}$ and $Y^{p,q'}$ are contact inequivalent when $p'\neq p$ using a certain crepant toric symplectic filling. In particular, they show that such fillings exist in dimension five, and that its equivariant symplectic homology is a contact invariant which essentially coincides with linearized contact homology. In the following $\phi$ denotes the Euler phi-function.

\begin{theorem}\label{ypqbouq}
The contact structures $Y^{p,q}$ and $Y^{p',q'}$ on $S^2\times S^3$ are inequivalent if and only if $p'\neq p$. Furthermore, the isotopy class of the contact structure defined by $Y^{p,1}$ admits a $\phi(p)$-bouquet $\gB_{\phi(p)}(Y^{p,q})$ such that each of the $\phi(p)$ Sasaki cones admits a unique Sasaki-Einstein metric, up to transverse biholomorphism. Moreover, these Einstein metrics are non-isometric as Riemannian metrics.
\end{theorem}

We view the bouquet $\gB_{\phi(p)}(Y^{p,q})$ as representing a component of the moduli space whose underlying toric contact structure is $\cald_p$.  

\subsection{Complexity 1 Sasakian Structures on Connected Sums of $S^2\times S^3$}
Next we consider contact structures on the connected sums $M=k(S^2\times S^3)$ with complexity 1, that is there is an effective action of a 2-dimensional torus on $(M,\cald)$, but no effective $\bbt^3$ action. The analysis is similar to the case of $S^5$ treated above, but without a complete classification for complexity 1. These are given on the link $L_{p,q,k}$ of the Brieskorn-Pham polynomial $z_0^p+z_1^q+z_2^2+z_3^2$ with $\gcd(p,q)=k$ and $p,q\geq 2$ where both\footnote{When $p=q=2$ we have the link of a quadric which is a homogeneous space, and hence, toric.} $p,q$ are not both $2$. The link $L_{p,q,k}$ is diffeomorphic to the $(k-1)$-fold connected sum $(k-1)(S^2\times S^3)$ where $k=1$ means $S^5$. The main result is that of Collins and Sz\'ekelyhidi \cite{CoSz15} which says that $L_{p,q,k}$ admits an SE metric if and only if $2p>q$ and $2q>p$. In \cite{BovC16} it was shown that if either of these inequalities is not satisfied, then the entire Sasaki cone can admit no extremal Sasaki metric, and in \cite{BMvK15} it was shown that in either case there are infinitely many components in the corresponding moduli space. We have 

\begin{theorem}\label{compl1s2s3thm}
The moduli space $\gM^{cl}_{0,+,1}(k(S^2\times S^3))$ of positive complexity 1 Sasakian classes on $k(S^2\times S^3)$ with $c_1(\cald_{p,q,k})=0$ satisfies for each $k\geq 1$
\begin{enumerate}
\item  $\gM^{cl}_{0,+,1}(k(S^2\times S^3))$ has an infinite number of components;
\item there are an infinite number of components that do not intersect a toric contact structure;
\item there is a Reeb vector field in the Sasaki cone of the link $L_{p,q,k}$ that gives an SE metric if and only if $2p>q$ and $2q>p$;
\item if either of these inequalities is violated the entire Sasaki cone admits no extremal Sasaki metric.
\end{enumerate}
\end{theorem}

\begin{proof}
The mean Euler characteristic for $L_{p,q,k}$ was computed in \cite{BMvK15} to be 
\begin{equation}\label{mEulerpqk}
\chi_m(L_{p,q,k})=\frac{pq+k^2}{2(p+q)}
\end{equation}
which proves (1). To prove (2) we note that Abreu and Macarini \cite{AbMa18} prove that the mean Euler characteristic of a Gorenstein toric contact structure on $k(S^2\times S^3)$ is one half an integer which comparing with Equation \eqref{mEulerpqk} proves (2). As mentioned above item (3) is a result in \cite{CoSz15} and item (4) a result in \cite{BovC16}.
\end{proof}

In \cite{Suss18} S\"uss uses the theory of polyhedral divisors on the affine variety $Y$ to prove the existence of moduli of irregular SE metrics on the connected sums $(2k+1)(S^2\times S^3)$ for $k>1$. In this case there is a continuous family of $T$-equivalent Sasaki cones belonging to inequivalent CR structures.

Finally, we refer to \cite{BG05} and references therein for the existence of complexity 2 Sasakian structures as well as the moduli of SE metrics on the connected sums $k(S^2\times S^3)$. Of course, complexity 2 Sasaki CR structures have a 1-dimensional Sasaki cone and are necessarily quasiregular.

\subsection{Moduli of Sasakian Structures on $S^3$ Bundles over Riemann Surfaces of genus $g>0$}\label{Riesurfsect}
It is well known that topologically there are precisely two $S^3$ over a Riemann surface $\grS_g$ with structure group $O(4)$. They are the trivial bundle $\grS_g\times S^3$ and the non-trivial bundle $\grS_g\tilde{\times} S^3$, and they are distinguished by their second Stiefel-Whitney class. In each case we fix an orientation. In this section I give a brief discussion, without proofs, of the results found in \cite{BoTo11,BoTo13,BHLT16}. We should note that more complete results concerning the contact equivalence problem have been obtained only in the case of the trivial bundle. For it is only in that case that we show can, by using equivariant Gromov-Witten invariants  \cite{Bus10}, that there are complex structures $\{J_m\}$ with $m=0,\cdots,k-1$ on $\cald_k$ whose image under $\gQ$ are different, that is, they map to nonconjugate maximal tori in $\gC\go\gn(\Sigma_g\times S^3,\cald_k)$. Thus, although $(\cald_k,J_m)$ and $(\cald_k,J_{m'})$ are $S^1$ equivariantly equivalent, they are $\bbt^2$ equivariantly inequivalent. The contact inequivalence is determined by the classical invariant $c_1(\cald_k)$ for both the trivial and non-trival $S^3$ bundles, which is $2-2g-2k$ for $\Sigma_g\times S^3$ and $2-2g-2k-1$ for $\grS_g\tilde{\times} S^3$.
In what follows we consider $\bbt^2$ equivariant contact structures $\cald_k$ on $S^3$ bundles over $\grS_g$ and for simplicity we consider $g\geq 2$. On each $\cald_k$ we have a family of complex structures $\{J_{\grt,\grr,m}\}$ where $\grt\in\gM_g$ the moduli space of complex structures on $\grS_g$, $\grr\in{\rm Pic}^0(\grS_g)\approx T^{2g}$ the Jacobian torus, and $m=0,\cdots,k-1$ coming from reducible representation of $\pi_1(M)$ which give 2-dimensional Sasaki cones $\grk(\cald_k,J_{\grt,\grr,m})$. We also have complex structures coming from the irreducible representations which are parameterized by the smooth part $\calr^{irr}(\grS_g)$ of the character variety $\calr(\grS_g)$ which by Narasimhan and Seshadri \cite{NaSe65} correspond to stable rank 2 holomorphic vector bundles on $\grS_g$. These give 1-dimensional Sasaki cones $\grk(\cald_k,J_{\grt,\grr',0})$. However, it is only in the case of $\Sigma_g\times S^3$ that we have proved the existence of bouquets.
Note that the Sasaki cones are 1-dimensional for irreducible representations and 2-dimensional for reducible representation. These spaces could be non-Hausdorff owing to jumping phenomenon. A similar bouquet can be given in the genus one case. Summarizing we have 

\begin{theorem}\label{sasbouqthm}
The manifolds $\Sigma_g\times S^3$ and $\grS_g\tilde{\times} S^3$ admit a countably infinite number of contact inequivalent $\bbt^2$ equivariant contact structures $\{\cald_k\}_{k=1}^\infty$ of Sasaki type. Moreover, every $\bbt^2$ equivariant contact structure on $\Sigma_g\times S^3$ and $\grS_g\tilde{\times} S^3$ is contactomorphic to one of the $\cald_k$s. The contact structure $\cald_k$ also has a 1-dimensional Sasaki cone $\grk(\cald_k,J_{\grt,\grr',0})$ for each $\grt\in\gM_g$ and each irreducible representation $\grr'$ of $\pi_1(M)$. When $M=\Sigma_g\times S^3$ we can write the moduli space\footnote{This was incorrectly stated as the bouquet in \cite{BoTo14a,BHLT16}. The bouquet is just $\gB_k(\cald_k)=\bigcup_{m=0}^{k-1}\grk(\cald_k,J_{\grt,\grr,m})$ for each $\grt$ and $\grr$.} $\gM^{cl}_{\grg'}(\cald_k)$, where $c_1(\cald_k)=\grg'=(2-2g-2k)\grg$, in terms of Sasaki bouquets as
\begin{equation*}\label{Modsigmag}
\gM^{cl}_{\grg'}(\cald_k)=\bigcup_{m=0}^{k-1}\bigcup_{\grt\in\gM_g,\grr\in T^{2g}}\grk(\cald_k,J_{\grt,\grr,m})\bigcup_{\grt\in\gM_g,\grr'\in \calr^{irr}(\grS_g)}\grk(\cald_k,J_{\grt,\grr',0}).
\end{equation*}
\end{theorem}




We are interested in when the moduli spaces have extremal representatives, and in particular those of constant scalar curvature (CSC). It is well known that the stable case with the 1-dimensional Sasaki cones $\grk(\cald_k,J_{\grt,\grr',0})$ have CSC Sasaki representatives. So we focus our attention on the $\bbt^2$ equivariant case.

\begin{theorem}\label{sasexh}
Given any genus $g>0,\grt\in\gM_g,\grr\in T^{2g}$ and non-negative integer $m_{max}$, 
\begin{enumerate}
\item there exists a positive integer $K_{g,m_{max}}$ such that for all integers $k\geq K_{g,m_{max}}$ the  two dimensional Sasaki cones $\grk(\cald_k,J_{\grt,\grr,m})$ in the contact structure $\cald_k$ are exhausted by extremal Sasaki metrics for $m=0,\ldots,m_{max}$; 
In particular, if $0<g\leq 4$ the 2-dimensional Sasaki cones are exhausted by extremal Sasaki metrics for all $k\in\bbz^+$;
\item the 2-dimensional Sasaki cones $\grk(\cald_k,J_{\grt,\grr,m})$ admit a unique CSC Sasaki metric;
\item   in the case of the trivial bundle $\grS_g\times S^3$ when $2\leq g\leq 4$ the entire moduli space $\gM^{cl}_{\grg'}(\cald_k)$ is exhausted by extremal Sasaki metrics.
\end{enumerate}
\end{theorem} 

On the other hand if we let the genus grow, we can lose extremality. 

\begin{proposition}\label{deg>0intro}
For any choice of genus $g\geq 20$ there exist at least one choice of $(k,m)$ with $m=1,\ldots,k-1$ such that the regular ray in the Sasaki cone $\grk(\cald_k,J_{\grt,\grr,m})$ admits no extremal representative. 
\end{proposition}

In this case the regular ray is relatively K-unstable.


\section{Sasaki Moduli Spaces in higher dimension}
Much less is known about the Sasaki moduli spaces for dimension greater than 5. There are some rather sporatic results depending on the construction used. For example in the case of Brieskorn manifolds, or even more generally complete interesections of hypersurfaces of weighted homogeneous singularities many results can be obtained. However, for the sake of time as well as convenience we limit the discussion to several examples. Results for the cases not treated can be found in \cite{BG05,BoTo14P,BMvK15,BoTo14a} and references therein.

\subsection{Sasaki Moduli on Homotopy Spheres}
Here we restrict our discussion to homotopy spheres that bound parellelizable manifolds. It is known, cf. Theorem 9.5.10 in \cite{BG05}, that for any homotopy sphere $\grS^{2n+1}$ that bounds a parallelizable manifold $|\pi_0(\gM^{cl}_+(\grS^{2n+1}))|=\aleph_0$. When $n$ is odd the components belong to inequivalent almost contact structures, and when $n$ is even the result is due to Ustilovsky \cite{Ust99} using contact homology. The well definedness of the contact homology in this case has recently been established by Gutt \cite{Gutt17}. For further discussion of this result see Section 6.1 of \cite{KwvKo13}. In particular the moduli space $\gM^{cl}_+(\grS^{2n+1})$ of the homotopy spheres given in Table \ref{sasconetable} have an infinite number of components.

We now consider the SE moduli on homotopy spheres. These moduli spaces cannot belong to the classes $\gM^{cl}_+$ described by the homotopy spheres of Table \ref{sasconetable} since the latter admit no extremal metrics at all in its Sasaki cone. So the former belong to different components than those of Table \ref{sasconetable}. From \cite{BGK05,BGKT05,BMvK15} we have

\begin{theorem}\label{SEsphere}
Let $\grS^{4n+1}$ be a homotopy sphere bounding a parellelizable manifold. Then
\begin{enumerate}
\item the number of components of $\gM^{SE}(\grS^{4n+1})$ grows doubly exponentially with dimension;
\item for the standard diffeomorphism type on $S^9$ we have $|\pi_0(\gM^{SE}(S^9))|\geq 983$ and for the exotic $\grS^9$ we have $|\pi_0(\gM^{SE}(\grS^9))|\geq 494$. Moreover, the components belong to distinct components of \linebreak $\gM^{cl}_{+}(S^9)$ and $\gM^{cl}_{+}(\grS^9)$, respectively.
\end{enumerate}
\end{theorem}

For homotopy spheres of dimension $4n+3$ we only consider the case of homotopy 7-spheres. We present our results in Table \ref{SE7sph} from \cite{BMvK15}. In the table, `sig' is the Hirzebruch signature of the Stein filling $W$, which indicates the exotic nature of $\grS^7$ and $N$ is the number of Brieskorn homotopy $7$-spheres with the indicated signature. The third column gives the number of pairs of Brieskorn spheres of the indicated signature with the same mean Euler characteristic. There is also one case of a triple which occurs with signature 15. From this one easily obtains lower bounds on $|\pi_0(\gM^{SE}(\grS^7))|$. The table is based on an Excel file that can be found at Otto van Koert's webpage {http://www.math.snu.ac.kr/$\sim$okoert/}.


\begin{table}
\begin{center}
\begin{tabular}{| l | r | r || r | r | r |}
\hline
sig & $N$ & pairs & sig & N & pairs \\ \hline
0 & 353 &      0     & 14  & 390 & 1   \\ \hline
1 & 376 &     0      & 15  & 409 & 0   \\ \hline
2 & 336 & 2          & 16  & 352 & 3   \\ \hline
3 & 260 & 1          & 17  & 226 & 1   \\ \hline
4 & 294 & 1          & 18  & 260 & 0   \\ \hline
5 & 231 & 4          & 19  & 243 & 0   \\ \hline
6 & 284 & 2          & 20  & 309 & 1   \\ \hline
7 & 322 & 1          & 21  & 292 & 1   \\ \hline
8 & 402 & 2          & 22  & 425 & 1   \\ \hline
9 & 317 & 1          & 23  & 307 & 2   \\ \hline
10&309 & 5          & 24  & 298 & 0   \\ \hline
11&252 & 2          & 25  & 230 & 1   \\ \hline
12&304 & 0          & 26  & 307 & 2   \\ \hline
13&258 & 0          & 27  & 264 & 0   \\ \hline

\end{tabular}
\end{center}
\medskip
\caption{Oriented Homotopy $7$-spheres with SE metrics}
\label{SE7sph}
\end{table}

\subsection{Sasaki Moduli on Lens Space Bundles over Hodge Manifolds}
It follows from Proposition 7.6.7  of \cite{BG05} that any Sasaki join of the form $M_{\bfl,\bfw}=M\star_\bfl S^3_\bfw$ where $M$ is a regular Sasaki manifold over a compact Hodge manifold $N$ is the total space of a 3-dimensional lens space bundle over $N$ with fiber $S^3/\bbz_{l_2}$. The integers $\bfl=(l_1,l_2)$ and $\bfw=(w_1,w_2)$ satisfy the conditions of Section \ref{joinsect}. Clearly this case is a generalization of the previous cases treated in Sections \ref{s2s3sect} and \ref{Riesurfsect}. 

\begin{question}\label{lensbundtop}
What can one say about the topology of the manifolds $M_{\bfl,\bfw}$?
\end{question}

\begin{question}\label{lensbundquest}
What can one say about the set of contact structures of Sasaki type on the manifolds $M_{\bfl,\bfw}$?
\end{question}

The first question is much more accessible than the second. For as was seen in \cite{BoTo14a} if one knows the differentials for the Serre spectral sequence of the fibration $S^1\ra{1.8}M\ra{1.8} N$ one can, in principal, compute the cohomology ring of the manifold $M_{\bfl,\bfw}$. As for the second question the answer appears to depend strongly on the Hodge manifold $N$. In particular, we are interested in the cardinality $|\pi_0(\gM^{cl}_+(M_{\bfl,\bfw}))|$. We have an answer in a particular case \cite{BoTo14b} which generalizes previous work of Wang and Ziller \cite{WaZi90}. This is the homogeneous case $\bfw=(1,1)$ in our notation.

\begin{theorem}\label{lensbuntopthm}
Let $S^1\ra{1.8}S^{2p+1}\ra{1.8}\bbc\bbp^p$ be the standard regular Sasaki circle bundle. The join $M_{l_1,l_2,\bfw}=S^{2p+1}\star_{l_1,l_2}S^3_\bfw$ has
integral cohomology ring
$$H^*(M_{l_1,l_2,\bfw},\bbz)\approx\bbz[x,y]/(w_1w_2l_1^2x^2,x^{p+1},x^2y,y^2)$$
where $x,y$ are classes of degree $2$ and $2p+1$, respectively. Furthermore, for each homotopy type there is a finite number of diffeomorphism types; hence, for some diffeomorphism type there is a countable infinity of toric contact structures $\{\cald_k\}$ of Sasaki type which are inequivalent as contact structures.
\end{theorem}

\begin{proof}[Outline of proof]
The computation of the cohomology ring is standard and appears in \cite{BoTo14a,BoTo14b}. The finiteness of the diffeomorphism types is a result of the formality (in the sense of rational homotopy theory) of the manifold $S^{2p+1}\star_{l_1,l_2}S^3_\bfw$ with details appearing in \cite{BoTo14b}. It follows that for for each $l_1,w_1,w_2$ and some diffeomorphism type there is a countably infinite number of toric contact structures $\{\cald_k\}$ of Sasaki type. The first Chern class of $\cald_{l_1,l_2,\bfw}$ can easily be computed, viz.
\begin{equation}\label{c1}
c_1(\cald_{l_1,l_2,\bfw})=\bigl(l_2(p+1)-l_1|\bfw|\bigr)\grg
\end{equation}
which by varying $l_2$ implies that infinitely many $\{\cald_k\}$ are contact inequivalent.
\end{proof}

Since $\bbc\bbp^p$ is Fano, there are infinitely many with positive Sasakian structures, so we have $|\pi_0(\gM^{cl}_+(S^{2p+1}\star_{l_1,l_2}S^3_\bfw))|=\aleph_0$. More specific topological information about these structures in dimension 7 can be found in \cite{BoTo14b}.

Let us now turn to the moduli spaces of extremal, CSC, and SE structures where $M$ is any regular Sasaki manifold with constant scalar curvature.
The main results can be found in \cite{BoTo14a}. It is straightforward to generalize this by replacing the regularity condition on $M$ by that of quasi-regularity. This is currently being pursued \cite{BoTo18c}, so here we present the regular case.

\begin{theorem}\label{admjoincsc}
Let $M_{\bfl,\bfw}=M\star_{l_1,l_2}S^3_\bfw$ be the $S^3_\bfw$-join with a regular Sasaki manifold $M$ which is an $S^1$-bundle over a compact K\"ahler manifold $N$ with constant scalar curvature $s_N$. Then 
\begin{enumerate}
\item there exists a Reeb vector field $\xi_\bfv$ in the 2-dimensional $\bfw$-Sasaki cone on $M_{\bfl,\bfw}$ such that the corresponding ray of Sasakian structures $\cals_a=(a^{-1}\xi_\bfv,a\eta_\bfv,\Phi,g_a)$ has constant scalar curvature;
\item if $s_N\geq 0$, then the $\bfw$-Sasaki cone $\gt^+_\bfw$ is exhausted by extremal Sasaki metrics.
\item if $s_N>0$ then for sufficiently large $l_2$ there are at least three CSC rays in the $\bfw$-Sasaki cone of the join $M_{l_1,l_2,\bfw}$.
\end{enumerate}
\end{theorem}

For admissible $(l_1,l_2,\bfw)$ the map $\star_{l_1,l_2}S^3_\bfw$ on $\cals\calm$ induces a map on moduli spaces 
\begin{equation}\label{mapmod}
\gM^{cl}(M)\ra{2.8} \gM^{cl}(M\star_{l_1,l_2}S^3_\bfw).
\end{equation}
Item (1) of Theorem \ref{admjoincsc} says that the set $\gM^{CSC}$ of all Sasaki classes with a CSC representative is a submodule under $\star_{l_1,l_2}S^3_\bfw$.

To consider Sasaki-Einstein moduli from the $S^3_\bfw$ join operation we need to choose
$$l_{1}=\frac{\cali_N}{\gcd(w_1+w_2,\cali_N)},\qquad   l_2=\frac{w_1+w_2}{\gcd(w_1+w_2,\cali_N)},$$ where $\cali_N$ denotes the Fano index of $N$, which imposes the condition $c_1(\cald)=0$. If we assume that $N$ also has a positive K\"ahler-Einstein metric we can find an SE metric in the $\bfw$ Sasaki subcone $\gt^+_\bfw$ of the join $M_{l_1,l_2,\bfw}=M\star_{l_1,l_2}S^3_\bfw$ which gives rise to a nonempty moduli space $\gM^{SE}(M_{l_1,l_2,\bfw})$. These are straightforward multidimensional  genaralizations of Theorem \ref{ypqbouq}, so we fully expect the bouquet phenomenon. We can also mention that the every element of $\gt^+_\bfw$ can be represented by a Sasaki-Ricci soliton. 

\newcommand{\etalchar}[1]{$^{#1}$}
\def\cprime{$'$} \def\cprime{$'$} \def\cprime{$'$} \def\cprime{$'$}
  \def\cprime{$'$} \def\cprime{$'$} \def\cprime{$'$} \def\cprime{$'$}
  \def\cdprime{$''$} \def\cprime{$'$} \def\cprime{$'$} \def\cprime{$'$}
  \def\cprime{$'$}
\providecommand{\bysame}{\leavevmode\hbox to3em{\hrulefill}\thinspace}
\providecommand{\MR}{\relax\ifhmode\unskip\space\fi MR }
\providecommand{\MRhref}[2]{%
  \href{http://www.ams.org/mathscinet-getitem?mr=#1}{#2}
}
\providecommand{\href}[2]{#2}

\end{document}